\documentclass{jams-l}

\usepackage{amssymb}

\usepackage{dsfont}

\usepackage{mathabx}

\usepackage{graphicx}

\usepackage{hyperref}

\usepackage{xcolor}

\usepackage{mathbbol}

\usepackage{array} 

\usepackage{float}

\usepackage[autostyle]{csquotes} 

\newtheorem{theorem}{Theorem}[section]

\newtheorem{proposition}[theorem]{Proposition}
\newtheorem{lemma}[theorem]{Lemma}

\theoremstyle{definition}

\theoremstyle{remark}
\newtheorem{remark}[theorem]{Remark}

\numberwithin{equation}{section}

\DeclareSymbolFontAlphabet{\amsmathbb}{AMSb}%

\usepackage{tikz}

\usetikzlibrary{arrows,positioning,calc} 
\usetikzlibrary{decorations.pathreplacing}
\tikzset{
    >=stealth',
    punkt/.style={
           rectangle,
           rounded corners,
           draw=black, thick,
           text width=5.5em,
           minimum height=2em,
           text centered},
    punktl/.style={
           rectangle,
           rounded corners,
           draw=black, thick,
           text width=7em,
           minimum height=2em,
           text centered},
    pil/.style={
           ->,
           shorten <=4pt,
       shorten >=4pt
    },
    pildotted/.style={
           ->,
           shorten <=4pt,
           shorten >=4pt,
  dotted,
  },
    punktf/.style={
           rectangle,
           text width=4.0em,
           minimum height=1.5em,
           text centered},
    punktfTop/.style={
           rectangle,
           text width=4.0em,
           minimum height=1.5em,
           text centered,
           append after command={
               [thick,shorten >=0.2bp, shorten <=0.2bp]
               (\tikzlastnode.north west)edge(\tikzlastnode.north east)
}
    },
    punktfBot/.style={
           rectangle,
           text width=4.0em,
           minimum height=1.5em,
           text centered,
           append after command={
               [thick,shorten >=0.2bp, shorten <=0.2bp]
               (\tikzlastnode.south west)edge(\tikzlastnode.south east)
            }
    }
}

\usepackage{prodint}
\theoremstyle{plain}

\begin{document}

\title[Non-parametric estimators of scaled cash flows]{Non-parametric estimators of scaled cash flows}


\author{Theis Bathke}
\address{Institut für Mathematik, Fakultät V, Carl von Ossietzky Universität Oldenburg, Ammerländer Heerstraße 114-118, DE-26129 Oldenburg, Germany}
\curraddr{}
\thanks{}

\author{Christian Furrer}
\address{Department of Mathematical Sciences, University of Copenhagen, Universitetsparken 5, DK-2100 Copenhagen, Denmark}
\curraddr{}
\email{\href{mailto:furrer@math.ku.dk}{furrer@math.ku.dk}}
\thanks{}


\date{}

\dedicatory{}

\begin{abstract}

In multi-state life insurance, incidental policyholder behavior gives rise to expected cash flows that are not easily targeted by classic non-parametric estimators if data is subject to sampling effects. We introduce a scaled version of the classic Aalen--Johansen estimator that overcomes this challenge. Strong uniform consistency and asymptotic normality are established under entirely random right-censoring, subject to lax moment conditions on the multivariate counting process. In a simulation study, the estimator outperforms earlier proposals from the literature. Finally, we showcase the potential of the presented method to other areas of actuarial science. \\


\end{abstract}

\maketitle


\section{Introduction}\label{sec:intro}

In the mathematics of life insurance, multi-state modeling has been a canonical approach for more than half a {century}, owing back to the unification based on Markov chains in~\cite{Hoem1969}. Today, popular models in both theory and practice include Markov chain models~\cite{Norberg1991,MilbrodtStracke1997} and semi-Markov models~\cite{Hoem1972,JanssenDominicis1984,Helwich2008,Christiansen2012,BuchardtMollerSchmidt2015}, but also more sophisticated alternatives such as doubly stochastic and aggregate models~\cite{Buchardt2017,AhmadBladtFurrer2023} have found some use. Until recently, the associated actuarial literature focused primarily on valuation techniques and computational aspects, while statistical methodology played a somewhat lesser role.

Estimation in multi-state models is, generally speaking, a challenging undertaking, not least due to sampling effects such as right-censoring. In Markov models, the celebrated Nelson--Aalen and Aalen--Johansen estimators~\cite{Aalen1978,AalenJohansen1978} may be applied directly for non-parametric predictive purposes or instead for model diagnostics and performance assessment of (semi-)parametric alternatives. Recently, the idea of fully non-Markov modeling has gained traction -- not only in actuarial science~\cite{Christiansen2021,ChristiansenFurrer2022}, but also in biostatistics and other related fields~\cite{PutterSpitoni2016,MaltzahnEtAl2021,NiesslEtAl2023}. By imposing no assumption whatsoever on the intertemporal dependence structure of the state process, one seeks to reduce systematic risk and increase robustness. Simulation studies underpin that, if there is a substantial departure from Markovianity, a robust approach is clearly to be preferred~\cite{PutterSpitoni2016}.

Importantly, the Aalen--Johansen estimator remains a consistent and asymptotically normal estimator for occupation probabilities even when the Markov assumption is relaxed~\cite{DattaSatten2001,Glidden2002}. The landmark Aalen--Johansen estimator, which just invokes sub-sampling, extends this directly to transition probabilities~\cite{PutterSpitoni2016} and, if the insurance contract admits a so-called deterministic canonical cash flow representation, also to state-wise prospective reserves~\cite{Christiansen2021} based on the plug-in principle.

An insurance contract admits a deterministic canonical cash flow representation if the associated sojourn and transition payments are deterministic, see Definition~5.1 in~\cite{Christiansen2021} for the specifics. This assumption is often violated in practice, a key example being so-called scaled payments~\cite{ChristiansenFurrer2022}. Such payments arise naturally in the context of implicit policyholder options, especially in connection with free policy conversion and stochastic retirement, see~\cite{HenriksenEtAl2014,BuchardtMoller2015,BuchardtMollerSchmidt2015,GadNielsen2016,Furrer2022}. Here, the contractual payments are scaled by a factor depending on the exercise time(s) of the option(s), as to account for the fact that free policy conversion or early retirement should lead to reduced future benefits.

The actuarial literature currently offers two non-parametric approaches that are able to handle scaled payments in the non-Markov regime under entirely random right-censoring. The first is the one-dimensional estimator of~\cite{ChristiansenFurrer2022}, which is of usual Nelson--Aalen and Aalen--Johansen type, but uses auxiliary data based on a change of measure argument lifted from~\cite{Furrer2022}. The second is the two-dimensional estimator of~\cite{Bathke2024}, which is an extension of the Volterra estimator of P.J.\ Bickel for bivariate survival functions~\cite{Dabrowska1988,GillVanDerLaanWellner1995} to multi-state models. Both estimators possess certain deficiencies, however. The auxiliary data used by the estimator of~\cite{ChristiansenFurrer2022} comes with extra noise not originally present in the sample. The estimator of~\cite{Bathke2024} is computationally demanding and inherits the inferior tail behavior of the Volterra estimator.

This paper introduces {novel scaled (or weighted) one-dimensional Nelson--Aalen and Aalen--Johansen estimators that resolve} the aforementioned issues. We establish strong uniform consistency and asymptotic normality of the estimator{s}. En passant, we also establish asymptotic normality of the two-dimensional estimator from~\cite{Bathke2024}. Finally, we compare the finite sample performances of the estimators in a simulation study. Overall, the scaled estimator{s are} found to have the best performance.

Throughout, primary focus lies on scaled payments arising from the inclusion of policyholder options, which imposes a specific structure on the scaling. However, the theoretical insights do not really rely on this structure, and our results are therefore more broadly applicable, including to situations where scaling is due to discounting. In particular, this leads to an extension of Theorem~3.1 of~\cite{AhmadBladt2023} to the non-Markov regime. Further, it allows for a constructive and simultaneous definition of forward interest and forward transition rates, addressing lasting concerns in the actuarial literature~\cite{MiltersenPersson2005,Norberg2010,Buchardt2014,ChristiansenNiemeyer2015,Buchardt2017,BuchardtFurrerSteffensen2019}.

The paper is structured as follows. In Section~\ref{sec:pre}, we introduce the basic multi-state setup, including the scaled payment process. This is followed by Section~\ref{sec:scaled}, in which the novel scaled one-dimensional {Nelson--Aalen and Aalen--Johansen estimators are} introduced. Section~\ref{sec:val} focuses on representations of the expected cash flow and the resulting targets for estimation{, substantiating the relevance of the scaled estimator, while Section~\ref{sec:num} contains} the main results on consistency and asymptotic normality. In Section~\ref{sec:sim}, we present the comparative simulation study. Finally, Section~\ref{sec:extension} showcases the broader actuarial applicability of our approach.

\section{Preliminaries}\label{sec:pre}

Let $(\Omega,\mathcal{F},\amsmathbb{F},\amsmathbb{P})$ be an underlying filtered probability space. The state of the insured is modeled by a non-explosive jump process $Z = (Z_t)_{t\geq{0}}$ on a finite state space $\mathcal{J}$ with $Z_0 \equiv z_0 \in \mathcal{J}$. Note that $Z$ is not assumed to be say Markov or semi-Markov. We associate with $Z$ a multivariate counting process $N$, which components $N_{jk} = (N_{jk}(t))_{t\geq0}$ are given by
\begin{align*}
N_{jk}(t) = \#\big\{s \in (0,t] : Z_{s-} = j, Z_s = j\big\}.
\end{align*}
Throughout, we assume that $\amsmathbb{E}[N_{jk}(t)] < \infty$ for all components, which in particular implies the aforementioned non-explosion of $Z$. 

Incidental modeling of policyholder behavior such as free policy conversion and stochastic retirement proceeds by considering a state space of the form $\mathcal{J} = \mathcal{J}_0 \cup \mathcal{J}_1$, where $\mathcal{J}_0$ and $\mathcal{J}_1$ are the states pre- and post-exercise of the policyholder option, respectively. This interpretation is meaningful as long as we require $z_0 \in \mathcal{J}_0$ and $\mathcal{J}_1$ to be absorbing for $Z$, that is
\begin{align*}
\amsmathbb{P}(Z_t \in \mathcal{J}_1, Z_s \in \mathcal{J}_0) = 0 \text{ for } t \leq s.
\end{align*}
These assumptions are thus made throughout.

The first (and only) hitting time of $\mathcal{J}_1$ by $Z$ is denoted $\tau$:
\begin{align*}
\tau = \inf\{t > 0 : Z_t \in \mathcal{J}_1\}.
\end{align*}
We can interpret $\tau$ as the exercise time of the policyholder option.

Insurance contracts may be described by payment streams consisting of benefits less premiums. We are interested in contracts of the scaled form
\begin{align*}
B(\mathrm{d}t)
&=
H(t) B^\circ(\mathrm{d}t), &&B(0) = B^\circ(0), \\
H(t) &= \rho(\tau,Z_{\tau-},Z_\tau)^{\mathds{1}_{\{\tau \leq t\}}}, \\
B^\circ(\mathrm{d}t)
&=
\sum_{j\in\mathcal{J}} \mathds{1}_{\{Z_{t-} = j\}} B_j(\mathrm{d}t)
+
\sum_{j,k \in \mathcal{J}\atop j\neq k} b_{jk}(t)N_{jk}(\mathrm{d}t), &&B^\circ(0) \equiv b_0 \in \amsmathbb{R},
\end{align*}
where each $B_j$ is a right-continuous real function of finite variation, describing the deterministic sojourn payments in the respective state; each $b_{jk}$ is a measurable real function which is bounded on compacts, describing the deterministic transition payments between the respective states; and each scaling factor $[0,\infty) \ni t \mapsto \rho(t,j,k)$ is non-negative and predictable with respect to the information generated by $Z$, the filtration $\amsmathbb{F}^Z$. In practice, it is common to determine the scaling factors as to maintain actuarial equivalence on some safe-side basis and, due to the typically simple nature of that basis, each scaling factor would be deterministic rather than just predictable. 

In these contracts, all payments are scaled by a common factor upon exercise of the policyholder option. If the factor is below one, the payments are decreased, while if it is above one, the payments are increased. The idea is that, for instance, early retirement should lead to reduced retirement benefits, while delayed retirement should lead to increased retirement benefits. Throughout, we assume that each $\rho(\cdot,j,k)$ is uniformly bounded on compacts. However, the specific form of $B$ in terms of $B^\circ$ and $\rho$ has a redundancy and thus allows one to reparametrize the model such that each scaling factor $\rho(\cdot,j,k)$ is uniformly bounded by one.

Solely to decrease the notational burden, we assume that $\amsmathbb{P}(\tau = t) B_j(\mathrm{d}t) = 0$. In other words, lump sum sojourn payments may not occur simultaneously with the exercise of the policyholder option. If such lump sum sojourn payments are possible, they might always be included as transition payments from $\mathcal{J}_0$ to $\mathcal{J}_1$ instead.

In case of a finite maximal contract time $T\in(0,\infty)$, the prospective reserve at contract inception reads
\begin{align*}
&V(0-) = b_0 + V(0), &&V(0) = \int_0^T \frac{1}{\kappa(t)} A(\mathrm{d}t),
\end{align*}
for a deterministic savings account $\kappa$ assumed right-continuous, of finite variation, and bounded away from zero, and where the expected accumulated cash flow $A$ is defined according to $A(t) = \amsmathbb{E}[B(t) - B(0)]$.

In this paper, we focus on non-parametric estimation of $V(0-)$ through $A$. This might at first appear as a rather narrow approach. For instance, the as-if-Markov approach of~\cite{Christiansen2021,ChristiansenFurrer2022} considers expected accumulated cash flows and prospective reserves on the form
\begin{align*}
A(t_0,t) &= \amsmathbb{E}[B(t) - B(t_0) \, | Z_{t_0}], \\
V(t_0) &= \int_{t_0}^T \frac{\kappa(t_0)}{\kappa(t)} A(t_0,\mathrm{d}t).
\end{align*}
However, our approach may be adapted to also deal with this situation. The main idea is to use the landmark $Z_{t_0}$, meaning one should condition on or sub-sample with respect to events of the type $\{Z_{t_0} = i\}$, which is also the approach adopted in~\cite{Christiansen2021,ChristiansenFurrer2022}. We merely focus on the situation at contract inception to reduce the notational burden. If conditioning with respect to $Z_{t_0}$ is replaced by more sophisticated information such as $(Z_{t_0},U_{t_0})$, where $U$ is the duration process associated with $Z$, kernel methods could be used, confer with~\cite{BladtFurrer2024}. Kernel methods are outside the scope of this paper.

{

\section{Scaled estimators}\label{sec:scaled}

If data consisted of complete observations of $n$ independent insured, with the $\ell$'th insured represented by $Z^\ell$, we could {for the expected accumulated cash flow} use the estimator
\begin{align*}
[0,\infty) \ni t \mapsto 
\frac{1}{n} \sum_{\ell=1}^n {\big(B^\ell(t) - B^\ell(0)\big)},
\end{align*}
where $B^\ell$ refers to $B$, but with $Z$ replaced by $Z^\ell$ and similarly for the associated multivariate counting processes.

In this paper, we are interested in the situation where observation of $Z$ is right-censored. To this end, we introduce a strictly positive random variable $R$ describing right-censoring. Letting $\eta$ denote the (possible infinite) absorption time of $Z$, we thus observe the pair
\begin{align*}
\big((Z_t)_{0\leq t\leq R}, \eta \wedge R\big).
\end{align*}
By $(Z_t)_{0 \leq t \leq R}$ we understand the full path of the jump process stopped at $R$, meaning $Z_t$ for $t \leq R$ and $Z_R$ for $t > R$. Given the initial state, observation of the stopped jump process corresponds to observing all transitions in the random interval $[0,R]$. Note that the right-censoring time $R$ is only observed if absorption {does not occur} before right-censoring.

If $Z$ had been assumed Markov, non-parametric estimation of the (cumulative) transition rates and transition probabilities under various sampling effects is still possible using the classic Nelson--Aalen and Aalen--Johansen estimators, respectively. The Nelson--Aalen and Aalen--Johansen estimators actually remain viable for certain `averaged' hazards and for occupation probabilities even without any assumptions on the intertemporal dependence structure of $Z$, confer with~\cite{DattaSatten2001}. However, due to the scaled nature of the payment streams involved, the estimation of more complex quantities that transcend mere occupation probabilities is required. To overcome this challenge, we introduce scaled versions of the Nelson--Aalen and Aalen--Johansen estimators. Disregarding some technical nuisances, these estimators are given by
\begin{align*}
\mathbb{\Lambda}^\rho_{jk}(t)
&=
\int_0^t \frac{1}{\amsmathbb{I}^\rho_j(s-)}  \amsmathbb{N}^\rho_{jk}(\mathrm{d}s), \\
\mathbb{p}^\rho_j(t)
&=
\bigg[\Prodi_0^t \big(\text{Id} + \mathbb{\Lambda}^\rho(\mathrm{d}s)\big)\bigg]_{z_0j},
\end{align*}
where for $n$ i.i.d.\ replicates $\big((Z^\ell_t)_{0\leq t \leq R^\ell},\eta^\ell \wedge R^\ell\big)_{\ell=1}^n$ the components are given by
\begin{align*}
\amsmathbb{I}^\rho_j(t)&= 
\frac{1}{n} \sum_{\ell=1}^n H^\ell(t) \mathds{1}_{\{Z^\ell_t = j\}} \mathds{1}_{\{t < R^\ell\}}, \\
\amsmathbb{N}^\rho_{jk}(t) &= \frac{1}{n}\sum_{\ell=1}^n \int_0^t H^\ell(s) \mathds{1}_{\{s \leq R^\ell\}} N^\ell_{jk}(\mathrm{d}s).
\end{align*}
We recover the classic Nelson--Aalen and Aalen--Johansen estimators by taking $H \equiv 1$. If and when these estimators converge, their limits must be
\begin{align*}
\Lambda^\rho_{jk}(t) &= \int_0^t \frac{1}{\amsmathbb{E}\big[H(s-) \mathds{1}_{\{Z_{s-} = j\}}\big]} \, \mathrm{d}\amsmathbb{E}{\bigg[\int_0^s H(u) N_{jk}(\mathrm{d}u)\bigg]}, \\
p^\rho_j(t)
&=
\bigg[\Prodi_0^t \big(\text{Id} + \Lambda^\rho(\mathrm{d}s)\big)\bigg]_{z_0j},
\end{align*}
presupposing that the terms related to right-censoring in the denominator and numerator cancel. The outline for the next two sections is now as follows. In Section~\ref{sec:val}, we identify these limits and connect them to the expected accumulated cash flows, which validates the purposefulness of the scaled estimators. In Section~\ref{sec:num}, we examine in more detail the properties of the scaled estimators, proving their uniform strong consistency and asymptotic normality. Throughout, we also draw comparisons to and further expand on the alternative two-dimensional approach of~\cite{Bathke2024}.

\subsubsection*{Notational matters}

In the following, a considerable notational burden can unfortunately not be avoided. However, the reader can be reassured that the following overarching principles are applied:
\begin{itemize}
\item[--] Estimators are `blackboard bold' versions of their targets. For instance, the estimator $\mathbb{p}^\rho_j$ targets the theoretical quantity $p^\rho_j$ \\[-2mm]
\item[--] Occupation probabilities, mean counts, and similar quantities are called $p$; we distinguish between whether a single or two subscripts is used \\[-2mm]
\item[--] Hazards and related quantities are called $\Lambda$ \\[-2mm]
\item[--] Scaled quantities carry $\rho$ as a superscript \\[-2mm]
\item[--] The two-dimensional approach involves vectors, which are typeset in \textbf{bold}.
\end{itemize}

}

\section{Valuation}\label{sec:val}

{In this section, we seek} a representation of the expected accumulated cash flow with components that may naturally be targeted for non-parametric estimation even when observation is incomplete. This typically involves hazards.

\subsection{One-dimensional representations}\label{subsec:1d_reps}

The first representation is inspired by results in~\cite{ChristiansenFurrer2022}. The approach of~\cite{ChristiansenFurrer2022} utilizes the general change of measure techniques developed in~\cite{Furrer2022}. We offer a slight adjustment of this approach that is both more direct and avoids the condition that the scaling factors $\rho(\cdot,j,k)$ be uniformly bounded by one.

Define $p^\rho_j$, $p^\rho_{jk}$, $p_j$, and $p_{jk}$ according to
\begin{align*}
p^\rho_j(t) &= \amsmathbb{E}[H(t) \mathds{1}_{\{Z_t = j\}}], &&p_j(t) = \amsmathbb{E}[\mathds{1}_{\{Z_t = j\}}], &&j \in \mathcal{J}, \\
p^\rho_{jk}(t) &= \amsmathbb{E}\Big[\int_0^t H(s) N_{jk}(\mathrm{d}s)\Big], &&p_{jk}(t) = \amsmathbb{E}[N_{jk}(t)], &&j,k \in \mathcal{J}, j \neq k.
\end{align*}
Further, define $\Lambda^\rho_{jk}$ and $\Lambda_{jk}$ according to
\begin{align*}
\Lambda^\rho_{jk}(t) &= \int_0^t \frac{1}{p^\rho_j(s-)} p^\rho_{jk}(\mathrm{d}s),  &&\Lambda_{jk}(t) = \int_0^t \frac{1}{p_j(s-)} p_{jk}(\mathrm{d}s), &&j,k \in \mathcal{J}, j \neq k.
\end{align*}
\begin{remark}
Note that for $j,k \in \mathcal{J}_0$, $j \neq k$, it holds that $p^\rho_j = p_j$ and $p^\rho_{jk} = p_{jk}$ and therefore $\Lambda^\rho_{jk} = \Lambda_{jk}$. Furthermore, for $j \in \mathcal{J}_1, k \in \mathcal{J}_0$ it holds that $\Lambda^\rho_{jk} = \Lambda_{jk} = 0$.
\end{remark}
Throughout, we shall assume that $\Lambda_{jk}(t) < \infty$ and $\Lambda^\rho_{jk}(t) < \infty$. Since each $\rho(\cdot,j,k)$ is assumed uniformly bounded on compacts, the former implies the latter if also $\rho(\cdot,j,k)$ is uniformly bounded away from zero on the support of $\tau$.
\begin{remark}
That each $\rho(\cdot,j,k)$ is uniformly bounded away from zero on the support of $\tau$ is a sufficient but not necessary condition for $\Lambda_{jk}(t) < \infty$ to imply $\Lambda^\rho_{jk}(t) < \infty$. Let $j,k \in \mathcal{J}_1$, $j \neq k$, which is the interesting case. If for instance the compensator $C_{jk}$ of $N_{jk}$ is given by
\begin{align*}
C_{jk}(t) = \int_0^t \mathds{1}_{\{Z_{s-} = j\}} \alpha_{jk}(\mathrm{d}s),
\end{align*}
compare with the case of $Z$ being Markov, and $\rho(\cdot,j,k)$ is strictly positive, then
\begin{align*}
p^\rho_{jk}(\mathrm{d}t)
&=
\amsmathbb{E}[\rho(\tau,Z_{\tau-},Z_\tau)\amsmathbb{P}(Z_{t-} = j \, | \, Z_{\tau})]\alpha_{jk}(\mathrm{d}t) \\
&=
\amsmathbb{E}[\rho(\tau,Z_{\tau-},Z_\tau)\mathds{1}_{\{Z_{t-} = j\}}]\alpha_{jk}(\mathrm{d}t) 
=
p^\rho_j(t-) \alpha_{jk}(\mathrm{d}t),
\end{align*}
so that $\Lambda^\rho_{jk}(\mathrm{d}t) = \alpha_{jk}(\mathrm{d}t)$. In similar fashion, it holds that $\Lambda_{jk}(\mathrm{d}t) = \alpha_{jk}(\mathrm{d}t)$, so that $\Lambda^\rho_{jk} = \Lambda_{jk}$. Therefore, whether each $\rho(\cdot,j,k)$ is uniformly bounded away from zero on the support of $\tau$ or not, it holds that $\Lambda_{jk}(t) < \infty$ implies $\Lambda^\rho_{jk}(t) < \infty$.
\end{remark}
In the following, let $\Lambda_{jj}$, $j\in\mathcal{J}$, be such that all row sums are constantly equal to zero, that is $-\Lambda_{jj} = \sum_{k \neq j} \Lambda_{jk}$. To the contrary, let $\Lambda^\rho_{jj}$ be given by
\begin{align*}
-\Lambda^\rho_{jj} &= -\Lambda_{jj}, &&j \in \mathcal{J}_0, \\
-\Lambda^\rho_{jj} &= \sum_{k \in \mathcal{J} \atop k \neq j} \Lambda^\rho_{jk}, &&j \in \mathcal{J}_1.
\end{align*}
Note that all row sums $\sum_{k \in \mathcal{J}} \Lambda^\rho_{jk}$ are not necessarily constantly equal to zero, unless each $\rho(\cdot,j,k)$ is constantly equal to one.
\begin{proposition}\label{prop:rep1d}
It holds that
\begin{align*}
A(\mathrm{d}t)
=
\sum_{j\in\mathcal{J}} p^\rho_j(t-) \bigg( B_j(\mathrm{d}t) + \sum_{k \in \mathcal{J}\atop k \neq j} b_{jk}(t) \Lambda^\rho_{jk}(\mathrm{d}t)\bigg)
\end{align*}
and that
\begin{align*}
p^\rho_j(t)
&=
\bigg[\Prodi_0^t \big(\emph{Id} + \Lambda^\rho(\mathrm{d}s)\big)\bigg]_{z_0j}, &&j\in\mathcal{J}.
\end{align*}
\end{proposition}
\begin{remark}
Another way to phrase the second part of the theorem, which is also utilized in the proof, is that $(p_j^\rho)_j$ uniquely solves the forward integral equations
\begin{align*}
p_j^\rho(\mathrm{d}t) &= \sum_{k \in \mathcal{J}} p_k^\rho(t-) \Lambda^{\rho}_{kj}(\mathrm{d}t), &&j \in \mathcal{J},
\end{align*}
with initial conditions $p_j^\rho(0) = \mathds{1}_{\{j = z_0\}}$.
\end{remark}
\begin{proof}
The first result is trivial upon noticing that
\begin{align*}
p^\rho_j(t-)
=
\amsmathbb{E}[H(t) \mathds{1}_{\{Z_{t-} = j\}}]
+
\amsmathbb{E}[(1-\rho(\tau,Z_{\tau-},Z_\tau))\mathds{1}_{\{\tau = t\}} \mathds{1}_{\{Z_{t-} = j\}}],
\end{align*}
where the latter term vanishes since we assumed that $\amsmathbb{P}(\tau = t)B_j(\mathrm{d}t) = 0$.

The second result is more involved. Introduce
\begin{align*}
q_j(t) = \bigg[\Prodi_0^t \big(\text{Id} + \Lambda^\rho(\mathrm{d}s)\big)\bigg]_{z_0j}.
\end{align*}
These are the unique solution to a set of forward integral equations, which due to the specific choice of the diagonal of $\Lambda^\rho$ read
\begin{align*}
q_j(\mathrm{d}t) &= \sum_{k \in \mathcal{J}_0} q_k(t-) \Lambda_{kj}(\mathrm{d}t), &&j \in \mathcal{J}_0, \\
q_j(\mathrm{d}t) &= \sum_{k \in \mathcal{J}} q_k(t-) \Lambda^\rho_{kj}(\mathrm{d}t), &&j \in \mathcal{J}_1,  
\end{align*}
with boundary conditions $q_j(0) = \mathds{1}_{\{j = z_0\}}$ for $j \in \mathcal{J}_0$ and $q_j(0) = 0$ for $j \in \mathcal{J}_1$.

Recall, that
\begin{align}\label{eq:key_identity}
\mathds{1}_{\{Z_t = j\}}
=
\mathds{1}_{\{j = z_0\}}
+
\sum_{k \in \mathcal{J} \atop k \neq j} \big(N_{kj}(t) - N_{jk}(t)\big).
\end{align}
Let $j\in\mathcal{J}_0$. Then~\eqref{eq:key_identity} reads
\begin{align}\label{eq:key_identityJ0}
\mathds{1}_{\{Z_t = j\}}
=
\mathds{1}_{\{j = z_0\}}
+
\sum_{k \in \mathcal{J}_0 \atop k \neq j} N_{kj}(t)
-
\sum_{k \in \mathcal{J} \atop k \neq j} N_{jk}(t),
\end{align}
and since $p^\rho_j(t) = p_j(t)$, by taking expectations on both sides of~\eqref{eq:key_identityJ0} we obtain
\begin{align*}
p^\rho_j(\mathrm{d}t) &= \sum_{k \in \mathcal{J}_0} p_k(t-) \Lambda_{kj}(\mathrm{d}t), &&p^\rho_j(0) = \mathds{1}_{\{j = z_0\}}.
\end{align*}
Let instead $j\in\mathcal{J}_1$. Then~\eqref{eq:key_identity} implies
\begin{align}\label{eq:key_identityJ1}
H(t)\mathds{1}_{\{Z_t = j\}}
=
\sum_{k \in \mathcal{J} \atop k \neq j} \int_0^t H(s) N_{kj}(\mathrm{d}s)
-
\sum_{k \in \mathcal{J}_1 \atop k \neq j} \int_0^t H(s) N_{jk}(\mathrm{d}s).
\end{align}
Upon taking expectations on both sides of~\eqref{eq:key_identityJ1}, we finally obtain
\begin{align*}
p^\rho_j(\mathrm{d}t) &= \sum_{k \in \mathcal{J}} p^\rho_k(t-)\Lambda^\rho_{kj}(\mathrm{d}t), &&p^\rho_j(0) = 0.
\end{align*}
Collecting results and referring to the uniqueness of the solution to the forward integral equations, we conclude that $p^\rho_j = q_j$ for all $j \in \mathcal{J}$, as desired.
\end{proof}
\begin{remark}
The identification of $p^\rho$ with the product integral of $\Lambda^\rho$ is the non-trivial part of Proposition~\ref{prop:rep1d}. As a special case, we obtain that
\begin{align}\label{eq:overgaard}
p_j(t)
=
\bigg[\Prodi_0^t \big(\text{Id} + \Lambda(\mathrm{d}s)\big)\bigg]_{z_0j}.
\end{align}
Such identities were the focal point of~\cite{Overgaard2019}, who associated the left-hand side with a limit of products of transition probabilities and showed that if these transition probabilities are suitably regular, then $\Lambda_{jk}(t) < \infty$ and~\eqref{eq:overgaard} holds. By pertaining to the key identity~\eqref{eq:key_identity} and properties of the product integral, the proof becomes much shorter and simpler.
\end{remark}
The representation of $A$ in Proposition~\ref{prop:rep1d} is pleasing, since $A$ is essentially cast as a well-behaved map of $\Lambda^\rho$. However, the matrix function $\Lambda^\rho$ is -- strictly speaking -- not a hazard matrix. The presence of $H$ has entailed a choice of diagonal that does not ensure the row sums to be constantly zero. This differs from the approach in~\cite{ChristiansenFurrer2022}, where the introduction of an auxiliary jump process allows for a hazard interpretation. We briefly summarize this approach and compare it to our approach.

Suppose each $\rho(\cdot,j,k)$ is uniformly bounded by one, and define $\widetilde{Z} = (\widetilde{Z}_t)_{t\geq0}$ according to
\begin{align*}
\tilde{Z}_t
=
\mathds{1}_{\{Z_t \in \mathcal{J}_0\}} Z_t
+ \mathds{1}_{\{Z_t \in \mathcal{J}_1\}} \big(
\mathds{1}_{\{U \leq \rho(\tau,Z_{\tau-},Z_\tau)\}} Z_t + \mathds{1}_{\{U > \rho(\tau,Z_{\tau-},Z_\tau)\}} \nabla\big),
\end{align*}
where $\nabla$ is an artificial cemetery state and $U \sim \text{Unif}(0,1)$ is independent of $Z$. Set $\widetilde{\mathcal{J}} = \mathcal{J} \cup \{\nabla\}$, denote by $\widetilde{N}$ the multivariate counting process corresponding to $\widetilde{Z}$, and define 
\begin{align*}
\widetilde{p}_j(t) &= \amsmathbb{E}[\mathds{1}_{\{\tilde{Z}_t = j\}}], &&\widetilde{p}_{jk}(t) = \amsmathbb{E}[\widetilde{N}_{jk}(t)], &&&\widetilde{\Lambda}_{jk}(t) &= \int_0^t \frac{1}{\widetilde{p}_j(s-)} \widetilde{p}_{jk}(\mathrm{d}s)
\end{align*}
for $j,k \in \widetilde{\mathcal{J}}$, $j \neq k$. Further, set
\begin{align*}
-\widetilde{\Lambda}_{jj} = \sum_{k \in \widetilde{\mathcal{J}} \atop k \neq j}  \widetilde{\Lambda}_{jk}, &&j\in\widetilde{\mathcal{J}},
\end{align*}
so that all row sums are constantly equal to zero. In Subsection~4.1 of~\cite{ChristiansenFurrer2022} it is, among other things, shown that
\begin{align*}
A(\mathrm{d}t)
=
\sum_{j\in\mathcal{J}} \widetilde{p}_j(t-) \bigg( B_j(\mathrm{d}t) + \sum_{k \in \mathcal{J}\atop k \neq j} b_{jk}(t) \widetilde{\Lambda}_{jk}(\mathrm{d}t)\bigg)
\end{align*}
and that
\begin{align*}
\widetilde{p}_j(t)
&=
\bigg[\Prodi_0^t \big(\text{Id} + \widetilde{\Lambda}(\mathrm{d}s)\big)\bigg]_{z_0j}, &&j\in\mathcal{J}.
\end{align*}
From the latter identity it is rather obvious that $\Lambda^\rho$ corresponds exactly to $\widetilde{\Lambda}$ with the row and column corresponding to the artificial cemetery state removed. What is then the advantage of the latter representation? This representation is convenient insofar that it retains the interpretation of $\widetilde{\Lambda}$ as an actual hazard matrix and thus of $\widetilde{p}$ as an actual probability vector, from which estimation procedures and their asymptotic properties almost suggest themselves. The representation is inconvenient, on the other hand, insofar that it requires an auxiliary added source of randomness, namely $U$, as well as the condition that the scaling factors be uniformly bounded by one.

\subsection{Two-dimensional representation}

Two-dimensional representations were proposed and studied in~\cite{BathkeChristiansen2024}. We here give a concise summary adapted to the case of scaled payments. In comparison to the one-dimensional representation, two further assumptions are required, namely the strengthened moment condition $\amsmathbb{E}[N_{jk}(t)^2] < \infty$ and that $\rho(\cdot,j,k)$ is deterministic, where $j,k\in\mathcal{J}$, $j \neq k$.

Let $\boldsymbol{t} = (t_1,t_2)$ and $\boldsymbol{j} = (j_1,j_2)$. Define $p_{\boldsymbol{j}}$ and $p_{\boldsymbol{j}\boldsymbol{k}}$ according to
\begin{align*}
p_{\boldsymbol{j}}(\boldsymbol{t})
&=
\amsmathbb{E}[\mathds{1}_{\{Z_{t_1} = j_1\}}\mathds{1}_{\{Z_{t_2} = j_2\}}], &&j_1,j_2\in\mathcal{J}, \\
p_{\boldsymbol{j}\boldsymbol{k}}(\boldsymbol{t})
&=
\amsmathbb{E}[N_{j_1k_1}(t_1)N_{j_2k_2}(t_2)], &&j_1,k_1,j_2,k_2\in\mathcal{J}, j_1 \neq k_1, j_2 \neq k_2. 
\end{align*}
Further, define $\Lambda_{\boldsymbol{j}\boldsymbol{k}}$ according to
\begin{align*}
\Lambda_{\boldsymbol{j}\boldsymbol{k}}(\boldsymbol{t})
&=
\int_{\boldsymbol{0}}^{\boldsymbol{t}}
\frac{1}{p_{\boldsymbol{j}}(\boldsymbol{s-})} p_{\boldsymbol{j}\boldsymbol{k}}(\mathrm{d}\boldsymbol{s}),  &&j_1,k_1,j_2,k_2\in\mathcal{J}, j_1 \neq k_1, j_2 \neq k_2.
\end{align*}
Throughout, we shall assume that $\Lambda_{\boldsymbol{j}\boldsymbol{k}}(\boldsymbol{t}) < \infty$. The following result summarizes and adapts Example~2.4, Section~7, Theorem~3.3, and Theorem~4.3  from~\cite{BathkeChristiansen2024} to the current setting.
\begin{proposition}\label{prop: scaled-payments as two-dimensional cashflow}
Suppose that $\amsmathbb{E}[N_{jk}(t)^2] < \infty$ and that $\rho(\cdot,j,k)$ is deterministic for $j,k\in\mathcal{J}$, $j \neq k$. It then holds that
\begin{align*}
&A(t) - b_0 \\
&=
\sum_{j\in\mathcal{J}_0} \int_0^t p_j(s-)
\bigg(
B_j(\mathrm{d}s) + \sum_{k \in \mathcal{J}_0\atop k \neq j} b_{jk}(s) \Lambda_{jk}(\mathrm{d}s) + \sum_{k \in \mathcal{J}_1} \rho(s,j,k) b_{jk}(s) \Lambda_{jk}(\mathrm{d}s)
\bigg) \\
&\quad+
\sum_{j_2\in\mathcal{J}_1}\int_0^t \sum_{j_1 \in \mathcal{J}_0\atop k_1 \in \mathcal{J}_1} \int_{\boldsymbol{0}}^{(t,s-)} \rho(u_1,j_1,k_1)
\bigg( 
\sum_{k_2 \in \mathcal{J}\atop k_2 \neq j_2} p_{(j_1,k_2)}(\boldsymbol{u-}) \Lambda_{(j_1,k_2)(k_1,j_2)}(\mathrm{d}\boldsymbol{u})  \\
&\hspace{62mm}-
\sum_{k_2 \in \mathcal{J}_1\atop k_2 \neq j_2} p_{\boldsymbol{j}}(\boldsymbol{u}-) \Lambda_{\boldsymbol{j}\boldsymbol{k}}(\mathrm{d}\boldsymbol{u})
\bigg)
B_{j_2}(\mathrm{d}s) \\
&\quad+
\sum_{j_2,k_2 \in \mathcal{J}_1 \atop j_2 \neq k_2} \sum_{j_1\in\mathcal{J}_0\atop k_1\in\mathcal{J}_1} \int_{\boldsymbol{0}}^{(t,t)}
\rho(u_1,j_1,k_1)b_{j_2k_2}(u_2)p_{\boldsymbol{j}}(\boldsymbol{u-})\Lambda_{\boldsymbol{j}\boldsymbol{k}}(\mathrm{d}\boldsymbol{u})
\end{align*}
and that $(p_{\boldsymbol{j}})_{\boldsymbol{j}}$ uniquely solves the following system of inhomogeneous Volterra integral equations:
\begin{align*}
p_{\boldsymbol{j}}(\boldsymbol{t})
&=
\mathds{1}_{\{z_0 = j_1 = j_2\}}
+
\mathds{1}_{\{z_0 = j_1\}} \big(p_{j_2}(t_2) - \mathds{1}_{\{z_0 = j_2\}}\big)
+
\mathds{1}_{\{z_0 = j_2\}} \big(p_{j_1}(t_1) - \mathds{1}_{\{z_0 = j_1\}}\big) \\
&\quad+\hspace{-4mm}
\sum_{k_1,k_2\in\mathcal{J}\atop k_1 \neq j_1, k_2 \neq j_2}\hspace{-2mm}
\int_{\boldsymbol{0}}^{\boldsymbol{t}}
\Big(
p_{\boldsymbol{k}}(\boldsymbol{s-})\Lambda_{\boldsymbol{k}\boldsymbol{j}}(\mathrm{d}\boldsymbol{s})
-
p_{(k_1,j_2)}(\boldsymbol{s-})\Lambda_{(k_1,j_2)(j_1,k_2)}(\mathrm{d}\boldsymbol{s}) \\
&\hspace{28mm}-
p_{(j_1,k_2)}(\boldsymbol{s-})\Lambda_{(j_1,k_2)(j_2,k_1)}(\mathrm{d}\boldsymbol{s})
+
p_{\boldsymbol{j}}(\boldsymbol{s}-)\Lambda_{\boldsymbol{j}\boldsymbol{k}}(\mathrm{d}\boldsymbol{s})
\Big),
\end{align*}
where $\boldsymbol{j} = (j_1,j_2) \in \mathcal{J} \times \mathcal{J}$.
\end{proposition}
In comparing the one- and two-dimensional representations, we note that the two-dimensional representation needs further technical assumptions and is more involved, in particular from a numerical point of view. On the other hand, it achieves true separation between the probabilistic components and the contract elements, that is the payments and the scaling factors. This is convenient if one needs to calculate the expected accumulated cash flow across many different contract specifications. It also allows for a more clear separation between valuation and estimation in accordance with common actuarial workflows.

\section{Non-parametric estimation}\label{sec:num}

{Recall that we do not observe $(Z_t)_{0\leq t < \infty}$, but instead the pair $\big((Z_t)_{0\leq t\leq R}, \eta \wedge R\big)$, where $\eta$ is the (possible infinite) absorption time of $Z$ and $R$ describes right-censoring.} Under such an incomplete observation scheme, the behavior of $Z$ beyond the right endpoint of the support of $R$ of course remains elusive. In the following, we thus fix a $\theta$ strictly below this right endpoint and focus on the time interval $[0,\theta]$. The extension to the maximal interval is not pursued.

Throughout, we shall assume that right-censoring is entirely random, meaning that $Z$ is independent of $R$. This is a central assumption which can, essentially, only be weakened in case $Z$ is Markov. It implies that
\begin{align}\label{eq:lambda1d_c}
\Lambda^\rho_{jk}(t)
&=
\int_0^t \frac{1}{p^{\rho,\texttt{c}}_j(s-)} p^{\rho,\texttt{c}}_{jk}(\mathrm{d}s), &&j,k \in \mathcal{J}, j \neq k, \\ \label{eq:lambda2d_c}
\Lambda_{\boldsymbol{j}\boldsymbol{k}}(\boldsymbol{t})
&=
\int_{\boldsymbol{0}}^{\boldsymbol{t}}\frac{1}{p^{\texttt{c}}_{\boldsymbol{j}}(\boldsymbol{s-})} p^{\texttt{c}}_{\boldsymbol{j}\boldsymbol{k}}(\mathrm{d}\boldsymbol{s}), &&j_1,k_1,j_2,k_2\in\mathcal{J}, j_1 \neq k_1, j_2 \neq k_2,
\end{align}
for $t$, $t_1$, and $t_2$ below $\theta$ and with
\begin{align*}
p^{\rho,\texttt{c}}_j(t)
&=
\amsmathbb{E}[H(t)\mathds{1}_{\{Z_t = j\}}\mathds{1}_{\{t < R\}}],
&&p^{\rho,\texttt{c}}_{jk}(t)
=
\amsmathbb{E}\Big[\int_0^t H(s) \mathds{1}_{\{s \leq R\}} N_{jk}(\mathrm{d}s)\Big], \\
p^{\texttt{c}}_{\boldsymbol{j}}(\boldsymbol{t})
&=
\amsmathbb{E}[\mathds{1}_{\{Z_{t_1} = j_1\}}\mathds{1}_{\{Z_{t_2} = j_2\}}\mathds{1}_{\{\boldsymbol{t} < R\}}],
&&p^{\texttt{c}}_{\boldsymbol{j}\boldsymbol{k}}(\boldsymbol{t})
=
\amsmathbb{E}[N_{j_1k_1}(t_1 \wedge R)N_{j_2k_2}(t_2 \wedge R)].
\end{align*}
In the following, the non-scaled case of each $\rho(\cdot,j,k)$ being constantly equal to one is indicated by dropping the superscript $\rho$, which is also in accordance with previous definitions and notation.

The components in~\eqref{eq:lambda1d_c} and~\eqref{eq:lambda2d_c} are natural targets for estimation. However, the fact that the denominators might become zero is somewhat of a technical menace. We follow the suggestion of~\cite{BladtFurrer2024} to change the target of estimation by means of perturbation. That is, we instead consider
\begin{align*}
\Lambda^{\rho,\varepsilon}_{jk}(t)
&=
\int_0^t \frac{1}{p^{\rho,\texttt{c}}_j(s-) \vee \varepsilon} p^{\rho,\texttt{c}}_{jk}(\mathrm{d}s), &&j,k \in \mathcal{J}, j \neq k, \\
\Lambda^\varepsilon_{\boldsymbol{j}\boldsymbol{k}}(\boldsymbol{t})
&=
\int_{\boldsymbol{0}}^{\boldsymbol{t}}\frac{1}{p^{\texttt{c}}_{\boldsymbol{j}}(\boldsymbol{s-}) \vee \varepsilon} p^{\texttt{c}}_{\boldsymbol{j}\boldsymbol{k}}(\mathrm{d}\boldsymbol{s}), &&j_1,k_1,j_2,k_2\in\mathcal{J}, j_1 \neq k_1, j_2 \neq k_2,
\end{align*}
for some $\varepsilon>0$. We retain the choice of diagonal, that is we let $\Lambda^{\rho,\varepsilon}_{jj}$ be given by
\begin{align*}
-\Lambda^{\rho,\varepsilon}_{jj} &= \sum_{k \in \mathcal{J}\atop k \neq j} \Lambda^\varepsilon_{jk}, &&j \in \mathcal{J}_0, \\
-\Lambda^{\rho,\varepsilon}_{jj} &= \sum_{k \in \mathcal{J} \atop k \neq j} \Lambda^{\rho,\varepsilon}_{jk}, &&j \in \mathcal{J}_1.
\end{align*}
Furthermore, we set
\begin{align*}
p_j^{\rho,\varepsilon}(t) = \bigg[\Prodi_0^t \big(\text{Id} + \Lambda^{\rho,\varepsilon}(\mathrm{d}s)\big)\bigg]_{z_0j}, &&j\in\mathcal{J}.
\end{align*}
Finally, we define $(p^\varepsilon_{\boldsymbol{j}})_{\boldsymbol{j}}$  as the unique solution to the following system of inhomogeneous Volterra integral equations:
\begin{align}\label{eq:p2d_epsilon}
\begin{split}
p^\varepsilon_{\boldsymbol{j}}(\boldsymbol{t})
&=
\mathds{1}_{\{z_0 = j_1 = j_2\}} \\
&\quad+
\mathds{1}_{\{z_0 = j_1\}} \big(p^\varepsilon_{j_2}(t_2) - \mathds{1}_{\{z_0 = j_2\}}\big)
+
\mathds{1}_{\{z_0 = j_2\}} \big(p^\varepsilon_{j_1}(t_1) - \mathds{1}_{\{z_0 = j_1\}}\big) \\ 
&\quad+\hspace{-4mm}
\sum_{k_1,k_2\in\mathcal{J}\atop k_1 \neq j_1, k_2 \neq j_2}\hspace{-2mm}
\int_{\boldsymbol{0}}^{(t,t)}
\Big(
p^\varepsilon_{\boldsymbol{k}}(\boldsymbol{s-})\Lambda^\varepsilon_{\boldsymbol{k}\boldsymbol{j}}(\mathrm{d}\boldsymbol{s})
-
p^\varepsilon_{(k_1,j_2)}(\boldsymbol{s-})\Lambda^\varepsilon_{(k_1,j_2)(j_1,k_2)}(\mathrm{d}\boldsymbol{s}) \\ 
&\hspace{28mm}-
p^\varepsilon_{(j_1,k_2)}(\boldsymbol{s-})\Lambda^\varepsilon_{(j_1,k_2)(j_2,k_1)}(\mathrm{d}\boldsymbol{s})
+
p^\varepsilon_{\boldsymbol{j}}(\boldsymbol{s}-)\Lambda^\varepsilon_{\boldsymbol{j}\boldsymbol{k}}(\mathrm{d}\boldsymbol{s})
\Big),
\end{split}
\end{align}
where $\boldsymbol{j} = (j_1,j_2) \in \mathcal{J} \times \mathcal{J}$.

In the limit $\varepsilon \to 0$, we recover the original expressions, see also Theorem~5.10 in~\cite{Bathke2024}. {For the transfer of asymptotic results, see Remark~2 and preceding examples in~\cite{BladtFurrer2024}.}

The classic identity~\eqref{eq:key_identity} allows us to establish the following link between the one-dimensional components:
\begin{align}\label{eq:key_1dJ0}
p^{\rho,\texttt{c}}_j(t)
&=
\mathds{1}_{\{j = z_0\}}
-
C^\rho_j(t)
+
\sum_{k \in \mathcal{J}_0\atop k \neq j} \big(p^{\rho,\texttt{c}}_{kj}(t)-p^{\rho,\texttt{c}}_{jk}(t)\big)
-
\sum_{k \in \mathcal{J}_1} p^{\texttt{c}}_{jk}(t), &&j\in\mathcal{J}_0, \\ \label{eq:key_1dJ1}
p^{\rho,\texttt{c}}_j(t)
&=
-
C^\rho_j(t)
+
\sum_{k\in\mathcal{J}_0} p^{\rho,\texttt{c}}_{kj}(t)
+
\sum_{k\in\mathcal{J}_1\atop k \neq j} \big(p^{\rho,\texttt{c}}_{kj}(t) - p^{\rho,\texttt{c}}_{jk}(t)\big), &&j\in\mathcal{J}_1,
\end{align}
where $C^\rho_j$ is defined according to
\begin{align*}
C^\rho_j(t) &= \amsmathbb{E}[H(R) \mathds{1}_{\{Z_R = j\}} \mathds{1}_{\{R \leq t\}}], &&j\in\mathcal{J}.
\end{align*}
Similarly, in regards to the two-dimensional case we have
\begin{align}\nonumber
p_{\boldsymbol{j}}^{\texttt{c}}(\boldsymbol t)
&=
\mathds{1}_{\{z_0 = j_1 = j_2\}} -\mathds{1}_{\{z_0 = j_1 = j_2\}}\amsmathbb{E}[\mathds{1}_{\{R\leq t_1\vee R\leq t_2\}}]  \\ \nonumber
&\quad+\mathds{1}_{\{z_0=j_2\}}\Big(\sum_{k\in\mathcal{J}\atop k\neq j_1}p^{\texttt{c}}_{kj_1}(t_1)-p^\texttt{c}_{j_1k}(t_1)\Big)
+\mathds{1}_{\{z_0=j_1\}}\Big(\sum_{k\in\mathcal{J}\atop k\neq j_2}p^{\texttt{c}}_{kj_2}(t_2)-p^{\texttt{c}}_{j_2k}(t_2)\Big) \\ \label{eq:def_C_2d}
&\quad-\mathds{1}_{\{z_0=j_2\}}\Big(\sum_{k\in\mathcal{J}\atop k\neq j_1}p^{\texttt{c},1}_{kj_1}(\boldsymbol t)-p^{\texttt{c},1}_{j_1k}(\boldsymbol t)\Big)
-\mathds{1}_{\{z_0=j_1\}}\Big(\sum_{k\in\mathcal{J}\atop k\neq j_2}p^{\texttt{c},2}_{kj_2}(\boldsymbol t)-p^{\texttt{c},2}_{j_2k}(\boldsymbol t)\Big)\\ \nonumber
&\quad+\hspace{-2mm}\sum_{k_1,k_2\in\mathcal{J}\atop k_1\neq j_1,k_2\neq j_2}\hspace{-2mm}\Big(p^{\texttt{c}}_{\boldsymbol k\boldsymbol j}(\boldsymbol t)+p^{\texttt{c}}_{\boldsymbol j\boldsymbol k}(\boldsymbol t)-p^{\texttt{c}}_{(j_1k_2)(k_1j_2)}(\boldsymbol t)-p^{\texttt{c}}_{(k_1j_2)(j_1k_2)}(\boldsymbol t)\Big) \\ \nonumber
&\quad-\hspace{-2mm}\sum_{k_1,k_2\in\mathcal{J}\atop k_1\neq j_1,k_2\neq j_2}\hspace{-2mm}\Big(p^{\texttt{c},3}_{\boldsymbol k\boldsymbol j}(\boldsymbol t)+p^{\texttt{c},3}_{\boldsymbol j\boldsymbol k}(\boldsymbol t)-p^{\texttt{c},3}_{(j_1k_2)(k_1j_2)}(\boldsymbol t)-p^{\texttt{c},3}_{(k_1j_2)(j_1k_2)}(\boldsymbol t)\Big),
\end{align}
where $p_{kj}^{\texttt{c},1}(\boldsymbol t)=\amsmathbb{E}[N_{kj}(t_1\wedge R)\mathds{1}_{\{R\leq t_1\vee R\leq t_2\}}]$, $p_{kj}^{\texttt{c},2}(\boldsymbol t)=\amsmathbb{E}[N_{kj}(t_2\wedge R)\mathds{1}_{\{R\leq t_1\vee R\leq t_2\}}]$ and $p^{\texttt{c},3}_{\boldsymbol k\boldsymbol j}(\boldsymbol t)=\amsmathbb{E}[N_{k_1j_1}(t_1\wedge R)N_{k_2j_2}(t_2\wedge R)\mathds{1}_{\{R\leq t_1\vee R\leq t_2\}}]$.

\subsection{Estimators}

Consider {once again $n$} i.i.d.\ replicates $\big((Z^\ell_t)_{0\leq t \leq R^\ell},\eta^\ell \wedge R^\ell\big)_{\ell=1}^n$ and form the following empirical processes
\begin{align*}
\amsmathbb{I}^\rho_j(t)&= 
\frac{1}{n} \sum_{\ell=1}^n H^\ell(t) \mathds{1}_{\{Z^\ell_t = j\}} \mathds{1}_{\{t < R^\ell\}}, \\
\amsmathbb{N}^\rho_{jk}(t) &= \frac{1}{n}\sum_{\ell=1}^n \int_0^t H^\ell(s) \mathds{1}_{\{s \leq R^\ell\}} N^\ell_{jk}(\mathrm{d}s), \\
\amsmathbb{I}_{\boldsymbol{j}}(\boldsymbol{t}) &= \frac{1}{n} \sum_{\ell=1}^n  \mathds{1}_{\{Z^\ell_{t_1} = j_1\}}\mathds{1}_{\{Z^\ell_{t_2} = j_2\}} \mathds{1}_{\{\boldsymbol{t} < R^\ell\}}, \\
\amsmathbb{N}_{\boldsymbol{j}\boldsymbol{k}}(\boldsymbol{t}) &= \frac{1}{n} \sum_{\ell=1}^n N^\ell_{j_1k_1}(t_1 \wedge R^\ell) N^\ell_{j_2k_2}(t_2 \wedge R^\ell{)}
\end{align*}
as non-parametric estimators of $p^{\rho,\texttt{c}}_j$, $p^{\rho,\texttt{c}}_{jk}$, $p^{\texttt{c}}_{\boldsymbol{j}}$, and $p^{\texttt{c}}_{\boldsymbol{j}\boldsymbol{k}}$, respectively. Similar to~\eqref{eq:key_1dJ0} and~\eqref{eq:key_1dJ1}, we find the following link for the one-dimensional case:
\begin{align} \label{eq:key_1dJ0emp}
\amsmathbb{I}^\rho_j(t)
&=
\mathds{1}_{\{j = z_0\}}
-
\amsmathbb{C}_j^\rho(t)
+
\sum_{k \in \mathcal{J}_0\atop k \neq j} \big(\amsmathbb{N}^\rho_{kj}(t)-\amsmathbb{N}^\rho_{jk}(t)\big)
-
\sum_{k \in \mathcal{J}_1} \amsmathbb{N}_{jk}(t), &&j\in\mathcal{J}_0, \\ \label{eq:key_1dJ1emp}
\amsmathbb{I}^\rho_j(t)
&=
-
\amsmathbb{C}^\rho_j(t)
+
\sum_{k\in\mathcal{J}_0} \amsmathbb{N}^\rho_{kj}(t)
+
\sum_{k\in\mathcal{J}_1\atop k \neq j} \big(\amsmathbb{N}^\rho_{kj}(t) - \amsmathbb{N}^\rho_{jk}(t)\big), &&j\in\mathcal{J}_1,
\end{align}
where $\amsmathbb{C}^\rho_j$ is defined according to
\begin{align*}
\amsmathbb{C}^\rho_j(t) &= \frac{1}{n} \sum_{\ell=1}^n H^\ell(R^\ell) \mathds{1}_{\{Z^\ell_{R^\ell} = j\}} \mathds{1}_{\{R^\ell \leq t\}}, &&j\in\mathcal{J}.
\end{align*}
Similarly, in regards to the two-dimensional case we have
\begin{align}\nonumber
\amsmathbb{I}_{\boldsymbol{j}}(\boldsymbol t)
&= \mathds{1}_{\{z_0=j_1=j_2\}} \\ \nonumber
&\quad+\mathds{1}_{\{z_0=j_2\}}\Big(\sum_{k\in\mathcal{J}\atop k\neq j_1}\amsmathbb{N}_{kj_1}(t_1)-\amsmathbb{N}_{j_1k}(t_1)\Big) \\ \nonumber
&\quad+\mathds{1}_{\{z_0=j_1\}}\Big(\sum_{k\in\mathcal{J}\atop k\neq j_2}\amsmathbb{N}_{kj_2}(t_2)-\amsmathbb{N}_{j_2k}(t_2)\Big) \\ \label{eq:def_C_2f}
&\quad-\mathds{1}_{\{z_0=j_2\}}\Big(\sum_{k\in\mathcal{J}\atop k\neq j_1}\amsmathbb{N}^1_{kj_1}(\boldsymbol t)-\amsmathbb{N}^1_{j_1k}(\boldsymbol t)\Big) \\ \nonumber
&\quad-\mathds{1}_{\{z_0=j_1\}}\Big(\sum_{k\in\mathcal{J}\atop k\neq j_2}\amsmathbb{N}^2_{kj_2}(\boldsymbol t)-\amsmathbb{N}^2_{j_2k}(\boldsymbol t)\Big) \\ \nonumber
&\quad+\hspace{-2mm}\sum_{k_1,k_2\in\mathcal{J}\atop k_1\neq j_1,k_2\neq j_2}\hspace{-2mm}\Big(\amsmathbb{N}_{\boldsymbol k\boldsymbol j}(\boldsymbol t)+\amsmathbb{N}_{\boldsymbol j\boldsymbol k}(\boldsymbol t)-\amsmathbb{N}_{(j_1k_2)(k_1j_2)}(\boldsymbol t)-\amsmathbb{N}_{(k_1j_2)(j_1k_2)}(\boldsymbol t)\Big) \\ \nonumber
&\quad-\hspace{-2mm}\sum_{k_1,k_2\in\mathcal{J}\atop k_1\neq j_1,k_2\neq j_2}\hspace{-2mm}\Big(\amsmathbb{N}^3_{\boldsymbol k\boldsymbol j}(\boldsymbol t)+\amsmathbb{N}^3_{\boldsymbol j\boldsymbol k}(\boldsymbol t)-\amsmathbb{N}^3_{(j_1k_2)(k_1j_2)}(\boldsymbol t)-\amsmathbb{N}^3_{(k_1j_2)(j_1k_2)}(\boldsymbol t)\Big),
\end{align}
where the estimators are defined according to
\begin{align*}
\amsmathbb{N}_{jk}(t)
&=
\frac{1}{n}\sum_{\ell=1}^n    N^\ell_{jk}(t\wedge R^\ell), \\
\amsmathbb{N}_{jk}^1(\boldsymbol t)
&=\frac{1}{n}\sum_{\ell=1}^n N^\ell_{jk}(t_1\wedge R^\ell)\mathds{1}_{\{R^\ell\leq t_1\vee R^\ell\leq t_2\}}, \\
\amsmathbb{N}_{jk}^2(\boldsymbol t)
&=\frac{1}{n}\sum_{\ell=1}^n N^\ell_{jk}(t_2\wedge R^\ell)\mathds{1}_{\{R^\ell\leq t_1\vee R^\ell\leq t_2\}}, \\
\amsmathbb{N}_{\boldsymbol{jk}}^3(\boldsymbol t)
&=
\frac{1}{n}\sum_{\ell=1}^n N^\ell_{j_1k_1}(t_1\wedge R^\ell)N^\ell_{j_2k_2}(t_2\wedge R^\ell)\mathds{1}_{\{R^\ell\leq t_1\vee R^\ell\leq t_2\}}.
\end{align*} 
The above identities are used repeatedly in the proofs, but may also be helpful for the development of efficient implementations.

We are now ready to form the estimators of interest, which will be of Nelson--Aalen and Aalen--Johansen type. The first estimator $\mathbb{\Lambda}^{\rho,\varepsilon}$ reads
\begin{align*}
\mathbb{\Lambda}^{\rho,\varepsilon}_{jk}(t)
=
\int_0^t \frac{1}{\amsmathbb{I}^\rho_j(s-) \vee \varepsilon}  \amsmathbb{N}^\rho_{jk}(\mathrm{d}s)
\end{align*}
and targets $\Lambda^{\rho,\varepsilon}_{jk}$. We also define
\begin{align*}
-\mathbb{\Lambda}^{\rho,\varepsilon}_{jj} &= \sum_{k \in \mathcal{J}\atop k \neq j} \mathbb{\Lambda}^\varepsilon_{jk}, &&j \in \mathcal{J}_0, \\
-\mathbb{\Lambda}^{\rho,\varepsilon}_{jj} &= \sum_{k \in \mathcal{J} \atop k \neq j} \mathbb{\Lambda}^{\rho,\varepsilon}_{jk}, &&j \in \mathcal{J}_1,
\end{align*}
where $\mathbb{\Lambda}^\varepsilon_{jk}$ is the ordinary estimator, corresponding to $\mathbb{\Lambda}^{\rho,\varepsilon}_{jk}$ without scaling, that is $H \equiv 1$. The associated estimator $\mathbb{p}^{\rho,\varepsilon}_j$ reads
\begin{align*}
\mathbb{p}^{\rho,\varepsilon}_j(t)
=
\bigg[\Prodi_0^t \big(\text{Id} + \mathbb{\Lambda}^{\rho,\varepsilon}(\mathrm{d}s)\big)\bigg]_{z_0j}
\end{align*}
and targets $p^{\rho,\varepsilon}_j$. The next estimator $\mathbb{\Lambda}^\varepsilon_{\boldsymbol{j}\boldsymbol{k}}$ reads
\begin{align*}
\mathbb{\Lambda}^\varepsilon_{\boldsymbol{j}\boldsymbol{k}}(\boldsymbol{t})
=
\int_{\boldsymbol{0}}^{\boldsymbol{t}} \frac{1}{\amsmathbb{I}_{\boldsymbol{j}}(\boldsymbol{s-}) \vee \varepsilon} \amsmathbb{N}_{\boldsymbol{j}\boldsymbol{k}}(\mathrm{d}\boldsymbol{s})
\end{align*}
and targets $\Lambda^\varepsilon_{\boldsymbol{j}\boldsymbol{k}}$. The associated estimator $\mathbb{p}^\varepsilon_{\boldsymbol{j}}$, which targets $p^\varepsilon_{\boldsymbol{j}}$, is defined as the unique solution to~\eqref{eq:p2d_epsilon}, but with all the theoretical quantities replaced by their just introduced estimators.

The perturbation given by $\varepsilon$ only plays a role in the theoretical analyses of the estimators. In applications, where $n$ is of course fixed, one takes $\varepsilon$ sufficiently close to zero as to not affect the actual estimates, which actually corresponds to $\varepsilon = 0$. 

Let $0=t_0 < t_1 < t_2 < \cdots < t_M$ describe all observed event times in $[0,\theta]$. The estimators are all constant between these time points and, consequently, best described through difference equations. In regards to the one-dimensional case, it holds that $\mathbb{p}^{\rho,\varepsilon}_j(t_0) = \mathds{1}_{\{z_0 = j\}}$ and
\begin{align*}
\Delta\mathbb{p}^{\rho,\varepsilon}_j(t_{m+1})
=
\sum_{k \in \mathcal{J}} \mathbb{p}^{\rho,\varepsilon}_k(t_m) \Delta\mathbb{\Lambda}^{\rho,\varepsilon}_{kj}(t_{m+1}).
\end{align*}
For the two-dimensional case, it holds that
\begin{align*}
\mathbb{p}^\varepsilon_{\boldsymbol{j}}(t_m, 0)
&=
\mathbb{p}^\varepsilon_{j_1}(t_m)\mathds{1}_{\{j_2 = z_0\}},
&&\mathbb{p}^\varepsilon_{\boldsymbol{j}}(0,t_m)
=
\mathbb{p}^\varepsilon_{j_2}(t_m)\mathds{1}_{\{j_1 = z_0\}},
\end{align*}
where $\mathbb{p}^\varepsilon_j$ is the ordinary estimator, corresponding to $\mathbb{p}^{\rho,\varepsilon}_j$ without scaling, that is $H \equiv 1$. Further, if we let
\begin{align*}
\Delta \mathbb{p}^\varepsilon_{\boldsymbol{j}}(t_{m_1+1},t_{m_2+1})
&=
\mathbb{p}^\varepsilon_{\boldsymbol{j}}(t_{m_1+1},t_{m_2+1}) \\
&\quad- \mathbb{p}^\varepsilon_{\boldsymbol{j}}(t_{m_1+1},t_{m_2}) - \mathbb{p}^\varepsilon_{\boldsymbol{j}}(t_{m_1},t_{m_2+1}) + \mathbb{p}^\varepsilon_{\boldsymbol{j}}(t_{m_1},t_{m_2}),
\end{align*}
we obtain
\begin{align*}
\Delta \mathbb{p}^\varepsilon_{\boldsymbol{j}}(t_{m_1+1},t_{m_2+1})
&=
\hspace{-2mm} \sum_{k_1,k_2\in\mathcal{J}\atop k_1 \neq j_1, k_2 \neq j_2} \hspace{-2mm} \mathbb{p}^\varepsilon_{\boldsymbol{k}}(t_{m_1},t_{m_2}) \Delta \mathbb{\Lambda}^{\varepsilon}_{\boldsymbol{k}\boldsymbol{j}}(t_{m_1+1},t_{m_2+1}) \\
&\quad-\hspace{-2mm} \sum_{k_1,k_2\in\mathcal{J}\atop k_1 \neq j_1, k_2 \neq j_2} \hspace{-2mm}\mathbb{p}^\varepsilon_{(k_1,j_2)}(t_{m_1},t_{m_2}) \Delta \mathbb{\Lambda}^{\varepsilon}_{(k_1,j_2)(j_1,k_2)}(t_{m_1+1},t_{m_2+1}) \\
&\quad-\hspace{-2mm} \sum_{k_1,k_2\in\mathcal{J}\atop k_1 \neq j_1, k_2 \neq j_2} \hspace{-2mm}\mathbb{p}^\varepsilon_{(j_1,k_2)}(t_{m_1},t_{m_2}) \Delta \mathbb{\Lambda}^{\varepsilon}_{(j_1,k_2)(k_1,j_2)}(t_{m_1+1},t_{m_2+1}) \\
&\quad-\mathbb{p}^\varepsilon_{\boldsymbol{j}}(t_{m_1},t_{m_2})\hspace{-2mm} \sum_{k_1,k_2\in\mathcal{J}\atop k_1 \neq j_1, k_2 \neq j_2} \hspace{-2mm}\Delta \mathbb{\Lambda}^{\varepsilon}_{\boldsymbol{j}\boldsymbol{k}}(t_{m_1+1},t_{m_2+1})
\end{align*}
with $\mathbb{\Lambda}^\varepsilon_{\boldsymbol{j}\boldsymbol{k}}(t_{m_1}, 0) = 0$, $\mathbb{\Lambda}^\varepsilon_{\boldsymbol{j}\boldsymbol{k}}(0,t_{m_2}) = 0$, and
\begin{align*}
\Delta\mathbb{\Lambda}^\varepsilon_{\boldsymbol{j}\boldsymbol{k}}(t_{m_1+1},t_{m_2+1})
=
\frac{1}{\amsmathbb{I}_{\boldsymbol{j}}(t_{m_1},t_{m_2}) \vee \varepsilon} \Delta \amsmathbb{N}_{\boldsymbol{j}\boldsymbol{k}}(t_{m_1+1},t_{m_2+1}),
\end{align*}
where $\Delta \amsmathbb{N}_{\boldsymbol{j}\boldsymbol{k}}$ reads
\begin{align*}
\Delta \amsmathbb{N}_{\boldsymbol{j}\boldsymbol{k}}(t_{m_1+1},t_{m_2+1})
=
\frac{1}{n}\sum_{\ell = 1}^n \Delta N^\ell_{j_1k_1}(t_{m_1+1} \wedge R^\ell) \Delta N^\ell_{j_2k_2}(t_{m_2+1} \wedge R^\ell).
\end{align*}
To efficiently calculate $\mathbb{p}^\varepsilon_{\boldsymbol{j}}$ on the relevant two-dimensional grid, one calculates simultaneously across $\boldsymbol{j} \in \mathcal{J} \times \mathcal{J}$ the gnomons 
\begin{align*}
\big(\mathbb{p}^\varepsilon_{\boldsymbol{j}}(t_{m+s},t_{m+1}),\mathbb{p}^\varepsilon_{\boldsymbol{j}}(t_{m+1},t_{m+s})\big)_{s=1,2,\ldots,M-m},
\end{align*}
progressing through $m$ from zero to $M-1$ and thus filling in the square. Note that the edges of the gnomons may be calculated in parallel. 

\subsection{Strong uniform consistency}\label{subsec:strong_unif_cons}

In Section~5 of~\cite{Bathke2024}, the approach of Section~3 in~\cite{BladtFurrer2024} was adapted to the two-dimensional case. In particular, using the link~\eqref{eq:def_C_2d} and the corresponding link~\eqref{eq:def_C_2f}, the following result was established:
\begin{theorem}
{Let $\varepsilon>0$.} Suppose that $\amsmathbb{E}[N_{jk}(\theta)^2] < \infty$ for $j,k\in\mathcal{J}$, $j \neq k$. It then holds that
\begin{align*}
&\sup_{0\leq \boldsymbol{t} \leq \theta} \big|
\mathbb{\Lambda}^\varepsilon_{\boldsymbol{j}\boldsymbol{k}}(\boldsymbol{t})
-
\Lambda^\varepsilon_{\boldsymbol{j}\boldsymbol{k}}(\boldsymbol{t})
\big|
\overset{\text{a.s.}}{\to} 0, &&n \to \infty, \\
&\sup_{0\leq \boldsymbol{t} \leq \theta} \big|
\mathbb{p}^\varepsilon_{\boldsymbol{j}}(\boldsymbol{t})
-
p^\varepsilon_{\boldsymbol{j}}(\boldsymbol{t})
\big|
\overset{\text{a.s.}}{\to} 0, &&n \to \infty.
\end{align*}
\end{theorem}
The consistency of  $\mathbb{\Lambda}^{\rho,\varepsilon}_{jk}$ and $\mathbb{p}^{\rho,\varepsilon}_j$ may also be established using the approach of Section~3 in~\cite{BladtFurrer2024}. First note that
\begin{align}\label{eq:unif_cons_C}
&\sup_{0\leq t \leq \theta} \big| \amsmathbb{C}^\rho_j(t) - C^\rho_j(t) \big| \overset{\text{a.s.}}{\to} 0, &&n \to \infty.
\end{align}
This follows from the fact that the class of all uniformly bounded, monotone functions on the real line is Glivenko--Cantelli, confer with Theorem~2.4.1 and Theorem~2.7.5 in~\cite{VaartWellner1996}. The assumption that each $\rho(\cdot,j,k)$ is uniformly bounded on $[0,\theta]$ is here of essence.
\begin{theorem}\label{thm:rho_consistency}
{Let $\varepsilon>0$.} It holds that
\begin{align*}
&\sup_{0\leq t \leq \theta} \big|
\mathbb{\Lambda}^{\rho,\varepsilon}_{jk}(t)
-
\Lambda^{\rho,\varepsilon}_{jk}(t)
\big|
\overset{\text{a.s.}}{\to} 0, &&n \to \infty, \\
&\sup_{0\leq t \leq \theta} \big|
\mathbb{p}^{\rho,\varepsilon}_j(t)
-
p^{\rho,\varepsilon}_j(t)
\big|
\overset{\text{a.s.}}{\to} 0, &&n \to \infty.
\end{align*}
\end{theorem}
{
\begin{proof}
See Subsection~\ref{sec:proofs}.
\end{proof}}

\subsection{Asymptotic normality}\label{subsec:as_norm}

To establish asymptotic normality, here meaning weak convergence to a tight Gaussian process, we use more or less explicitly the so-called bracketing central limit theorem, Theorem~2.11.9 in~\cite{VaartWellner1996}; see however also Appendix~D in~\cite{BladtFurrer2024}. In both cases, we first prove weak convergence of the relevant components and then adapt the approach of Section~4 in~\cite{BladtFurrer2024} to obtain weak convergence of the composite quantities. We begin with the one-dimensional representation. 

\begin{lemma}\label{lemma:comp_asn}
Suppose that $\amsmathbb{E}[N_{jk}(\theta)^2] < \infty$ for $j,k\in\mathcal{J}$, $j \neq k$. It then holds that
\begin{align*}
\Big(
\sqrt{n}\big(\amsmathbb{I}^\rho_j(t) - p^{\rho,\texttt{c}}_j(t)\big)_j,
\sqrt{n}\big(\amsmathbb{N}^\rho_{jk}(t) - p^{\rho,\texttt{c}}_{jk}(t)\big)_{j \neq k}
\Big)_{t\in[0,\theta]}
\end{align*}
converges weakly to a centered tight Gaussian process.
\end{lemma}
{
\begin{proof}
See Subsection~\ref{sec:proofs}.
\end{proof}}
\begin{theorem}\label{thm:as_nom_1d}
{Let $\varepsilon>0$.} Suppose that $\amsmathbb{E}[N_{jk}(\theta)^2] < \infty$ for $j,k\in\mathcal{J}$, $j \neq k$. It then holds that
\begin{align*}
\Big(
\sqrt{n}\big(\mathbb{p}^{\rho,\varepsilon}_j(t) - p^{\rho,\varepsilon}_j(t)\big)_j,
\sqrt{n}\big(\mathbb{\Lambda}^{\rho,\varepsilon}_{jk}(t) - \Lambda^{\rho,\varepsilon}_{jk}(t)\big)_{j\neq k}
\Big)_{t\in[0,\theta]}
\end{align*}
converges weakly to a centered tight Gaussian process.
\end{theorem}
{
\begin{proof}
See Subsection~\ref{sec:proofs}.
\end{proof}}
We now turn our attention to the two-dimensional case. We follow the strategy outlined in Remark~5.{13} of~\cite{Bathke2024}.
\begin{lemma}\label{lemma:comp_2d_asn}
Suppose that $\amsmathbb{E}[N_{jk}(\theta)^4] < \infty$ for $j,k\in\mathcal{J}$, $j \neq k$. It then holds that
\begin{align*}
\Big(
\sqrt{n}\big(\amsmathbb{I}_{\boldsymbol{j}}(\boldsymbol{t}) - p^{\texttt{c}}_{\boldsymbol{j}}(\boldsymbol{t})\big)_{\boldsymbol{j}},
\sqrt{n}\big(\amsmathbb{N}_{\boldsymbol{j}\boldsymbol{k}}(\boldsymbol{t}) - p^{\texttt{c}}_{\boldsymbol{j}\boldsymbol{k}}(\boldsymbol{t})\big)_{\boldsymbol{j} \neq \boldsymbol{k}}
\Big)_{\boldsymbol{t}\in[0,\theta] \times [0,\theta]}
\end{align*}
converges weakly to a centered tight Gaussian process.
\end{lemma}
{
\begin{proof}
See Subsection~\ref{sec:proofs}.
\end{proof}}
\begin{theorem}\label{thm:as_nom_2d}
{Let $\varepsilon>0$.} Suppose that $\amsmathbb{E}[N_{jk}(\theta)^4] < \infty$ for $j,k\in\mathcal{J}$, $j \neq k$. It then holds that
\begin{align*}
\Big(
\sqrt{n}\big(\mathbb{p}^\varepsilon_{\boldsymbol{j}}(\boldsymbol{t}) - p^\varepsilon_{\boldsymbol{j}}(\boldsymbol{t})\big)_{\boldsymbol{j}},
\sqrt{n}\big(\mathbb{\Lambda}^\varepsilon_{\boldsymbol{j}\boldsymbol{k}}(\boldsymbol{t}) - \Lambda^\varepsilon_{\boldsymbol{j}\boldsymbol{k}}(\boldsymbol{t})\big)_{\boldsymbol{j}\neq\boldsymbol{k}}
\Big)_{\boldsymbol{t}\in[0,\theta]\times[0,\theta]}
\end{align*}
converges weakly to a centered tight Gaussian process.
\end{theorem}
{
\begin{proof}
See Subsection~\ref{sec:proofs}.
\end{proof}}

\subsection{Discussion and remarks}

Theorems~\ref{thm:as_nom_1d} and~\ref{thm:as_nom_2d} are of qualitative nature. In principle, it is possible to calculate the covariances and, based hereon, propose plug-in estimators. This was done in~\cite{Glidden2002} for the one-dimensional case without scaling, and the scaled case should follow along similar lines. For the two-dimensional case, the calculations become very involved. In general, we suggest instead the use of the bootstrap. Its validity must be verified using, for instance, the (bootstrap) functional delta method and related results, see also Theorem~5.1 in~\cite{GillVanDerLaanWellner1995} for the validity of the bootstrap for bivariate survival functions. 

The results of Subsections~\ref{subsec:strong_unif_cons} and~\ref{subsec:as_norm} concern occupation probabilities and hazards. In practice, we are interested in expected accumulated cash flows and prospective reserves. If, however, we assume the transition payments $b_{jk}$ to be of finite variation, then the convergences immediately carry over to $\varepsilon$-perturbed versions of the quantities of interest by the continuous mapping theorem and functional delta method, respectively, and these quantities converge to the true quantities of interest as $\varepsilon \to 0$, see also Section~6 in~\cite{Bathke2024}.

{In the later simulation study,} we in particular compare the one-dimensional estimators of~\cite{ChristiansenFurrer2022} with the one-dimensional estimators of this paper. Therefore, we briefly introduce the former. 

The estimators of~\cite{ChristiansenFurrer2022} are based on the results at the end of Subsection~\ref{subsec:1d_reps}. They correspond to the usual Nelson--Aalen and Aalen--Johansen estimators (for occupation probabilities only), but using the data
\begin{align*}
\big((\tilde{Z}^\ell_t)_{0\leq t \leq R^\ell},\tilde{\eta}^\ell \wedge R^\ell\big)_{\ell=1}^n
\end{align*}
with $\tilde{Z}_t^\ell = \mathds{1}_{\{Z_t^\ell \in \mathcal{J}_0\}} Z_t^\ell + \mathds{1}_{\{Z_t^\ell \in \mathcal{J}_1\}} \big(
\mathds{1}_{\{U^\ell \leq \rho(\tau^\ell,Z^\ell_{\tau^\ell-},Z^\ell_{\tau^\ell})\}} Z_t^\ell + \mathds{1}_{\{U^\ell > \rho(\tau^\ell,Z^\ell_{\tau^\ell-},Z^\ell_{\tau^\ell})\}} \nabla\big)$, where $(U_\ell)_{\ell = 1}^n$ is an i.i.d.\ sequence of $\text{Unif}(0,1)$-distributed random variables that is independent of $(Z^\ell,R^\ell)_{\ell = 1}^n$. The resulting estimators are of the same type as the one-dimensional estimators studied in this paper, but without scaling. Their asymptotic theory and numerical implementation is already well understood, but might then also be derived from the results obtained here. However, the usage of the auxiliary variables $(U_\ell)_{\ell = 1}^n$ introduces additional and, as we substantiate in the {comparative simulation study}, unnecessary noise. 
	
{
\subsection{Proofs} \label{sec:proofs}

\begin{proof}[Proof of Theorem~\ref{thm:rho_consistency}]
Following for instance the proof of Proposition~{2} in~\cite{BladtFurrer2024}, it suffices to show strong uniform consistency of the components $\amsmathbb{I}^\rho_j$ and $\amsmathbb{N}^\rho_{jk}$. However, from the links~\eqref{eq:key_1dJ0}-\eqref{eq:key_1dJ1}, the links~\eqref{eq:key_1dJ0emp}-\eqref{eq:key_1dJ1emp}, and the strong uniform consistency~\eqref{eq:unif_cons_C}, it actually even just suffices to establish strong uniform consistency of $\amsmathbb{N}^\rho_{jk}$.

To this end, we use an idea from the proof of Proposition~{1} in~\cite{BladtFurrer2024}. Let $D_m$ be the event where $\amsmathbb{N}_{jk}(\theta) \leq m$ occurs eventually (in $n$). Recall that each $\rho(\cdot,j,k)$ is uniformly bounded on $[0,\theta]$ and note also that $p^{\rho,\texttt{c}}_{jk}$ is non-decreasing. Since the class of all uniformly bounded, monotone functions on the real line is Glivenko--Cantelli, we may conclude that on $D_m$ it holds that
\begin{align*}
&\sup_{0\leq t \leq \theta} \big| \amsmathbb{N}^\rho_{jk}(t) - p^{\rho,\texttt{c}}_{jk}(t) \big| \overset{\text{a.s.}}{\to} 0, &&n \to \infty.
\end{align*}
In addition, by the strong law of large numbers it holds that
\begin{align*}
&\amsmathbb{N}^\rho_{jk}(\theta)
\overset{\text{a.s.}}{\to} p^{\rho,\texttt{c}}_{jk}(\theta), &&n \to \infty,
\end{align*}
which shows that for $m$ sufficiently large $D_m$ has probability one. Collecting results completes the proof.
\end{proof}
	
\begin{proof}[Proof of Lemma~\ref{lemma:comp_asn}]
It suffices to show weak convergence of
\begin{align*}
\Big(
\sqrt{n}\big(\amsmathbb{C}^\rho_j(t) - C^\rho_j(t)\big)_j,
\sqrt{n}\big(\amsmathbb{N}^\rho_{jk}(t) - p^{\rho,\texttt{c}}_{jk}(t)\big)_{j \neq k}
\Big)_{t\in[0,\theta]}
\end{align*}
to a tight Gaussian process. This is because
\begin{align*}
\Big(
\sqrt{n}\big(\amsmathbb{I}^\rho_j(t) - p^{\rho,\texttt{c}}_j(t)\big)_j,
\sqrt{n}\big(\amsmathbb{N}^\rho_{jk}(t) - p^{\rho,\texttt{c}}_{jk}(t)\big)_{j \neq k}
\Big)_{t\in[0,\theta]}
\end{align*}
is a linear transformation hereof, confer with~\eqref{eq:key_1dJ0}-\eqref{eq:key_1dJ1} and~\eqref{eq:key_1dJ0emp}-\eqref{eq:key_1dJ1emp}, so it retains the weak convergence to a tight Gaussian process. (The limiting distribution is of course centered.) Convergence of the finite-dimensional distributions follows from the multivariate central limit theorem. According to Problem~1.5.3 in~\cite{VaartWellner1996}, it thus suffices to show asymptotic tightness of each component. Note that this is the case for $\sqrt{n}\big(\amsmathbb{C}^\rho_j(t) - C^\rho_j(t)\big)_{t\in[0,\theta]}$, given that the class of all uniformly bounded, monotone functions on the real line is Donsker, confer with Subsection~2.7.2 in~\cite{VaartWellner1996}. Further, under the assumptions $\amsmathbb{E}[N_{jk}(\theta)^2] {< \infty}$ this is also the case for $\sqrt{n}\big(\amsmathbb{N}^\rho_{jk}(t) - p^{\rho,\texttt{c}}_j(t)\big)_{t\in[0,\theta]}$, confer with Example~2.11.16 in~\cite{VaartWellner1996}. Collecting results completes the proof.
\end{proof}

\begin{proof}[Proof of Theorem~\ref{thm:as_nom_1d}]
Since the functional $(\Lambda^{\rho,\varepsilon}_{jk})_{j\neq k} \mapsto (p^{\rho,\varepsilon}_j)_j$ is Hadamard differentiable, confer with Theorem~8 in~\cite{GillJohansen1990}, it suffices by the functional delta method to establish weak convergence of the former. 

Let $t \mapsto X_j^{\rho,\texttt{c}}(t)$ and $t \mapsto X_{jk}^{\rho,\texttt{c}}(t)$ be the limit processes of $t \mapsto \sqrt{n}\big(\amsmathbb{I}^\rho_j(t) - p^{\rho,\texttt{c}}_j(t)\big)$ and $t \mapsto \sqrt{n}\big(\amsmathbb{N}^\rho_{jk}(t) - p^{\rho,\texttt{c}}_{jk}(t)\big)$, respectively. According to Lemma~\ref{lemma:comp_asn}, these limit processes are centered tight joint Gaussian. Similar to the proof of Theorem~{2} in~\cite{BladtFurrer2024}, we may show that $t \mapsto \sqrt{n}\big(\mathbb{\Lambda}^{\rho,\varepsilon}_{jk}(t) - \Lambda^{\rho,\varepsilon}_{jk}(t)\big)$ converges to the limit process
\begin{align}\label{eq:limit_process}
\begin{split}
&t \mapsto
\int_0^t  \frac{1}{p_j^{\rho,\texttt{c}}(s-)\vee \varepsilon} X_{jk}^{\rho,\texttt{c}}(\mathrm{d}s)
-
\int_0^t \frac{\mathds{1}_{\{p_j^{\rho,\texttt{c}}(s-)>\varepsilon\}}X_j^{\rho,\texttt{c}}(s-)}{p_j^{\rho,\texttt{c}}(s-)} \Lambda_{jk}^{\rho,\varepsilon}(\mathrm{d}s) \\
&=
\frac{X_{jk}^{\rho,\texttt{c}}(t)}{p_j^{\rho,\texttt{c}}(t)\vee \varepsilon}
-
\int_0^t X_{jk}^{\rho,\texttt{c}}(s) \, \mathrm{d}\bigg(\frac{1}{p_j^{\rho,\texttt{c}}(s)\vee \varepsilon}\bigg)
-
\int_0^t \frac{\mathds{1}_{\{p_j^{\rho,\texttt{c}}(s-)>\varepsilon\}}X_j^{\rho,\texttt{c}}(s-)}{p_j^{\rho,\texttt{c}}(s-)} \Lambda_{jk}^{\rho,\varepsilon}(\mathrm{d}s).
\end{split}
\end{align}
This limit process is again a centered tight Gaussian process. (In principle, one needs to verify that the finite-dimensional distributions are Gaussian, but this follows directly from the above decomposition.) To extend the result to $(\Lambda^{\rho,\varepsilon}_{jk})_{j\neq k}$, it suffices to verify convergence of the finite-dimensional distributions, confer with Problem~1.5.3 in~\cite{VaartWellner1996}. However, this follows immediately by the Cramer--Wold device in combination with~\eqref{eq:limit_process} and the joint Gaussianity from Lemma~\ref{lemma:comp_asn}.
\end{proof}

\begin{proof}[Proof of Lemma~\ref{lemma:comp_2d_asn}]
Similar to the one-dimensional case, confer with the proof of Lemma~\ref{lemma:comp_asn}, the link~\eqref{eq:def_C_2d} and the link~\eqref{eq:def_C_2f} imply that it suffices to establish asymptotic tightness of the components involved. In the following, we argue that the empirical process
\begin{align*}
[0,\theta] \times [0,\theta] \ni \boldsymbol{t} \mapsto \sqrt{n}\big(\amsmathbb{N}_{\boldsymbol{j}\boldsymbol{k}}(\boldsymbol{t}) - p^{\texttt{c}}_{\boldsymbol{j}\boldsymbol{k}}(\boldsymbol{t})\big)
\end{align*}
is asymptotically tight. For the other components, the argument is either simpler or quite similar and therefore, for the sake of brevity, omitted. We shall employ the bracketing central limit theorem with the envelope function
\begin{align*}
F
=
N_{j_1k_1}(\theta)N_{j_2k_2}(\theta),
\end{align*}
which is square integrable by assumption. To this end, let the $\mu$-bracketing number be denoted by $N_{[\,]}(\mu,\mathcal{P},L_2)$. This is the minimum number of sets $N_\mu$ in an $n$-independent partition $\mathcal{P} = \bigcup_{m=1}^{N_\mu} \mathcal{P}_{\mu,m}$ of the index set $[0,\theta] \times [0,\theta]$ such that, for every partitioning set $\mathcal{P}_{\mu,m}$,
\begin{align*}
\amsmathbb{E}
\Big[\sup_{\boldsymbol{t},\boldsymbol{s} \in \mathcal{P}_{\mu,m}} 
\big\lvert 
N_{j_1k_1}(t_1 \wedge R) N_{j_2k_2}(t_2 \wedge R) - N_{j_1k_1}(s_1 \wedge R)N_{j_2k_2}(s_2 \wedge R)
\big\rvert^2
\Big]
\leq \mu^2.
\end{align*}
Since the envelope function is square integrable by assumption, it suffices to identify an $n$-independent partition that confirms the entropy condition
\begin{align*}
\int_0^{\delta_n}
\sqrt{\log{N_{[\,]}(\mu,\mathcal{P},L_2)}} \, \mathrm{d}\mu \to 0, \quad \text{for every } \delta_n \downarrow 0.
\end{align*}
Note that if $\mathcal{P}_{\mu,m}$ is a rectangle $[\underline{t}_m, \overline{t}_m] \times [\underline{s}_m, \overline{s}_m]$, then by applying the equality $ab - cd = a(b-d) + d(a-c)$, we may bound
\begin{align*}
\sup_{\boldsymbol{t},\boldsymbol{s} \in \mathcal{P}_{\mu,m}} 
\big\lvert 
N_{j_1k_1}(t_1 \wedge R) N_{j_2k_2}(t_2 \wedge R) - N_{j_1k_1}(s_1 \wedge R)N_{j_2k_2}(s_2 \wedge R)
\big\rvert^2
\end{align*}
by twice the sum of
\begin{align*}
&\amsmathbb{E}\big[N_{j_1k_1}(\theta)^2 N_{j_2k_2}(\theta)\big(N_{j_2k_2}(\overline{s}_m \wedge R) - N_{j_2k_2}(\underline{s}_m\wedge R)\big)\big] \\
&+
\amsmathbb{E}\big[N_{j_1k_1}(\theta) N_{j_2k_2}(\theta)^2\big(N_{j_1k_1}(\overline{t}_m \wedge R) - N_{j_1k_1}(\underline{t}_m\wedge R)\big)\big]
\end{align*}
due to the monotonicity of the processes involved. For each of these terms, the argument of Example~2.11.16 in~\cite{VaartWellner1996} applies. In particular, there exists an $n$-independent partition $\mathcal{P}$ of at most $4 \amsmathbb{E}[N_{j_1k_1}(\theta)^2N_{j_2k_2}(\theta)^2] \mu^{-2}$ rectangles such that, for every rectangle $\mathcal{P}_{\mu,m}$,
\begin{align*}
\amsmathbb{E}
\Big[\sup_{\boldsymbol{t},\boldsymbol{s} \in \mathcal{P}_{\mu,m}} 
\big\lvert 
N_{j_1k_1}(t_1 \wedge R) N_{j_2k_2}(t_2 \wedge R) - N_{j_1k_1}(s_1 \wedge R)N_{j_2k_2}(s_2 \wedge R)
\big\rvert^2
\Big]
\leq \mu^2.
\end{align*}
Consequently, the entropy condition is satisfied. To see this, consider $\delta_n \downarrow 0$. For $n$ sufficiently large and an appropriate constant $K>0$,
\begin{align*}
\int_0^{\delta_n}
\sqrt{\log{N_{[\,]}(\mu,\mathcal{P},L_2)}} \, \mathrm{d}\mu
&\leq
\int_0^{\delta_n}
\sqrt{\log\!\big\{4 \amsmathbb{E}[N_{j_1k_1}(\theta)^2N_{j_2k_2}(\theta)^2] \mu^{-2}\big\}} \, \mathrm{d}\mu \\
&\leq
K \int_0^{\delta_n} \sqrt{\mu^{-1}} \, \mathrm{d}\mu \\
&=
2K \sqrt{\delta_n} \to 0.
\end{align*}
This completes the proof.
\end{proof}

\begin{proof}[Proof of Theorem~\ref{thm:as_nom_2d}]
The solution to the integral equations~\eqref{eq:p2d_epsilon} may be expressed with the help of a Peano series, see Lemma~A.1 in~\cite{Bathke2024}, which is a Hadamard differentiable functional, see Proposition~3.4 in~\cite{GillVanDerLaanWellner1995}. Therefore, it suffices by the functional delta method to establish weak convergence of
\begin{align*}
\Big(
\sqrt{n}\big(\mathbb{\Lambda}^\varepsilon_{\boldsymbol{j}\boldsymbol{k}}(\boldsymbol{t}) - \Lambda^\varepsilon_{\boldsymbol{j}\boldsymbol{k}}(\boldsymbol{t})\big)_{\boldsymbol{j}\neq\boldsymbol{k}}
\Big)_{\boldsymbol{t}\in[0,\theta]\times[0,\theta]}.
\end{align*}
Let $\boldsymbol{t} \mapsto X_{\boldsymbol{j}}^{\texttt{c}}(\boldsymbol{t})$ and $\boldsymbol{t} \mapsto X_{\boldsymbol{j}\boldsymbol{k}}^{\texttt{c}}(\boldsymbol{t})$ be the limit processes of $\boldsymbol{t} \mapsto \sqrt{n}\big(\amsmathbb{I}_{\boldsymbol{j}}(\boldsymbol{t}) - p^{\texttt{c}}_{\boldsymbol{j}}(\boldsymbol{t})\big)$ and $\boldsymbol{t} \mapsto \sqrt{n}\big(\amsmathbb{N}_{\boldsymbol{j}\boldsymbol{k}}(t) - p^{\texttt{c}}_{\boldsymbol{j}\boldsymbol{k}}(\boldsymbol{t})\big)$, respectively. Now, similar to the one-dimensional case, confer with the proof of Theorem~\ref{thm:as_nom_1d}, one may show that $\boldsymbol{t} \mapsto \sqrt{n}\big(\mathbb{\Lambda}^\varepsilon_{\boldsymbol{j}\boldsymbol{k}}(\boldsymbol{t}) - \Lambda^\varepsilon_{\boldsymbol{j}\boldsymbol{k}}(\boldsymbol{t})\big)$ converges to the limit process
\begin{align*}
\boldsymbol{t} \mapsto
\int_{\boldsymbol{0}}^{\boldsymbol{t}} \frac{1}{p_j^{\rho,\texttt{c}}(\boldsymbol{s}-)\vee \varepsilon} X_{\boldsymbol{j}\boldsymbol{k}}^{\texttt{c}}(\mathrm{d}\boldsymbol{s})
-
\int_{\boldsymbol{0}}^{\boldsymbol{t}} \frac{\mathds{1}_{\{p_{\boldsymbol{j}}^{\texttt{c}}(\boldsymbol{s}-)>\varepsilon\}}X_{\boldsymbol{j}}^{\texttt{c}}(\boldsymbol{s}-)}{p_{\boldsymbol{j}}^{\texttt{c}}(\boldsymbol{s}-)} \Lambda_{\boldsymbol{j}\boldsymbol{k}}^{\varepsilon}(\mathrm{d}\boldsymbol{s}),
\end{align*}
which is again a centered tight Gaussian process. (In principle, one needs to verify that the finite-dimensional distributions are Gaussian, but this follows by standard arguments applied to the above decomposition in combination with two-dimensional integration by parts.) The extension from fixed indices $\boldsymbol{j}$ and $\boldsymbol{k}$ to across indices transpires exactly as in the one-dimensional case.
\end{proof}

}

\section{Comparative simulation study}\label{sec:sim}

The setup of the simulation study is adopted from~\cite{Bathke2024}, the emphasize being on simplicity rather than authenticity. We consider a multi-state model with six states as in Figure~\ref{fig:state_space}, which allows for the modeling of the free policy as well as the surrender option. In this specific example, the jump process can at most be semi-Markov.

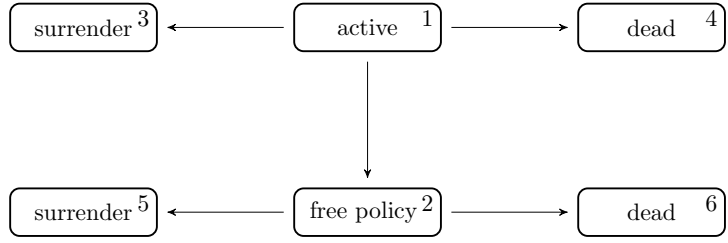
\begin{figure}[h]
	\centering
	\scalebox{0.9}{
	\begin{tikzpicture}[node distance=2em and 0em]
		\node[punkt] (2) {dead};
		\node[anchor=north east, at=(2.north east)]{$4$};
		\node[punkt, left = 20mm of 2] (0) {active};
		\node[anchor=north east, at=(0.north east)]{$1$};
		\node[punkt, left = 20mm of 0] (1) {surrender};
		\node[anchor=north east, at=(1.north east)]{$3$};
		\node[punkt, below = 20mm of 0] (3) {free policy\,\,};
		\node[anchor=north east, at=(3.north east)]{$2$};
		\node[punkt, left = 20mm of 3] (4) {surrender};
		\node[anchor=north east, at=(4.north east)]{$5$};
		\node[punkt, right = 20mm of 3] (5) {dead};
		\node[anchor=north east, at=(5.north east)]{$6$};		
	\path
		(0)	edge [pil]			node [above]			{}				(1)
		(0)	edge [pil]			node [above]			{}				(2)
		(0)	edge [pil]			node [above]			{}				(3)
		(3)	edge [pil]			node [above]			{}				(4)
		(3)	edge [pil]			node [above]			{}				(5)		
	;
	\end{tikzpicture}}
	\caption{Survival model with free policy and surrender options.}
	\label{fig:state_space}
\end{figure}

The insured has age $40$ years at contract inception. The contractual payments consist of an initial premium of $100,000$ monetary units, an annual premium rate while alive of $10,000$ monetary units until retirement at age $65$ years, and an annual benefit rate while alive from onset of retirement. Policyholder options are typically not explicitly included on the technical basis; this is achieved by determining the surrender payments and the free policy factor $t \mapsto \rho(t)$ such that there is no change in technical reserve upon exercise of the option. In particular, the technical reserve may be calculated in a survival model, and only a technical mortality rate and technical interest rate needs to be specified. {In practice, these rates are set prudently as part of the contract design and therefore not directly subject to statistical estimation{; in this simulation study, their only role is to determine the scaling factor}. We suppose that the technical mortality rate is
\begin{align*}
	\mu^\star(t) &= 0.005+10^{(5.728-10+0.038 t)},
\end{align*}
{which corresponds to the Danish G82 mortality table for females,} and that the technical interest rate is zero. Finally, we assume that the annual benefit rate has been determined to maintain actuarial equivalence under this technical basis. All in all, we then get
\begin{align*}
&B_1 (\mathrm{d}t) = \big(-10,000 \mathds{1}_{[0,25)}(t) + 22,658.67 \mathds{1}_{[25,\infty)}(t) \big) \, \mathrm{d}t, \\ 
&B_2 (\mathrm{d}t) = 22,658.67 \mathds{1}_{[25,\infty)}(t) \, \mathrm{d}t, \\
&b_0 = 100,000, \quad b_{13}(t) = V^{\star}(t), \quad b_{25}(t) = V^{\star,+}(t), \quad \rho(t) = \frac{V^\star(t)}{V^{\star,+}(t)},
\end{align*}
where $V^\star$ is the technical reserve, and $V^{\star,+}$ is the technical reserve of benefits solely. The resulting scaled payment stream reads
\begin{align*}
&B(\mathrm{d}t) \\
&= \rho(t)^{\mathds{1}_{\{\tau \leq t\}}} \Big(\mathds{1}_{\{Z_{t-} =1 \}} B_1(\mathrm{d}t) + \mathds{1}_{\{Z_{t-} = 2\}} B_2(\mathrm{d}t) + b_{13}(t) N_{13}(\mathrm{d}t) + b_{25}(t) N_{25}(\mathrm{d}t)\Big).
\end{align*}
The one-dimensional representation of the corresponding expected accumulated cash flow reads
\begin{align}\label{eq:1d_rep_sim_study}
\begin{split}
A(t)
=
b_0
+
\int_0^t \Big(&p_1(s-) B_1(\mathrm{d}s) + p_1(s-) b_{13}(s) \Lambda_{13}(\mathrm{d}s) \\
&+p_2^\rho(s-) B_2(\mathrm{d}s) + p_2^\rho(s-)b_{13}(s) \Lambda_{25}^\rho(\mathrm{d}s)\Big),
\end{split}
\end{align}
while the two-dimensional representation becomes
\begin{align}\label{eq:2d_rep_sim_study}
\begin{split}
A(t)
&=
\int_0^t \int_{\boldsymbol{0}}^{(t,s-)} \hspace{-3mm}\rho(u_1) \Big(p_{(1,1)}(\boldsymbol{u}-)\Lambda_{(1,1)(2,2)}(\mathrm{d}\boldsymbol{u}) \\
&\qquad\qquad\qquad\qquad\hspace{-2mm} - p_{(1,2)}(\boldsymbol{u}-)\big(\Lambda_{(1,2)(2,5)}(\mathrm{d}\boldsymbol{u}) + \Lambda_{(1,2)(2,6)}(\mathrm{d}\boldsymbol{u})\big)\Big) B_2(\mathrm{d}s) \\
&\quad+
b_0
+
\int_0^t p_1(s-) \Big(B_1(\mathrm{d}s) + b_{13}(s) \Lambda_{13}(\mathrm{d}s)\Big) \\
&\quad+
\int_{\boldsymbol{0}}^{(t,t)} \rho(u_1) p_{(1,2)}(\boldsymbol{u}-)b_{25}(u_2) \Lambda_{(1,2)(2,5)}(\mathrm{d}\boldsymbol{u}).
\end{split}
\end{align}
Data is simulated from a semi-Markov model with initial state $z_0 = 1$ and the following (duration-dependent) transition rates:
\begin{align*}
	&\mu_{14}(t) = \mu_{26}(t) = \mu^\star(t), \\
	&\mu_{12}(t) = 0.1\mathds{1}_{[0,25)}(t), \\
	&\mu_{13}(t) = 0.05\mathds{1}_{[0,25)}(t), \\
	&\mu_{25}(t,u) = \big(0.05+0.2\mathds{1}_{[\frac{1}{2},\frac{5}{2})}(u) \big) \mathds{1}_{[0,25)}(t).
\end{align*}
{This means, for instance, that
\begin{align*}
C_{25}(t) = \int_0^t \mathds{1}_{\{Z_{s-} = 2\}} \mu_{25}(s,U_{s-}) \, \mathrm{d}s,
\end{align*}
where $C_{25}$ is the compensator of $N_{25}$.} The right-censoring times are, solely for illustrative purposes, drawn from the distribution $\text{Unif}(20,80)$.

\subsection{Estimates} 

Recall that the perturbation given by $\varepsilon$ only plays a role in the theoretical analyses of the estimators. Correspondingly, throughout the simulation study, we set $\varepsilon$ equal to zero and suppress it notationally. True values from the underlying semi-Markov model are calculated using Monte Carlo methods based on an uncensored sample of size $n = 10,000$.

First, we compare our scaled Aalen--Johansen estimator of the scaled occupation probabilities in state $2$ with the proposed estimator of~\cite{ChristiansenFurrer2022}, which relies on change of measure arguments; we use the acronyms SAJ and CMAJ, respectively. The estimators are computationally similar, the complexity being $\mathcal{O}(n)$ for both. In Figure~\ref{fig:plot_probabilities}, we have plotted the resulting estimates for $n=2,000$ and, in the case of the CMAJ estimator, for three different realizations of its auxiliary uniformly distributed random variables.

\begin{figure}[h]
	\centering
	\includegraphics[width=0.9\textwidth]{"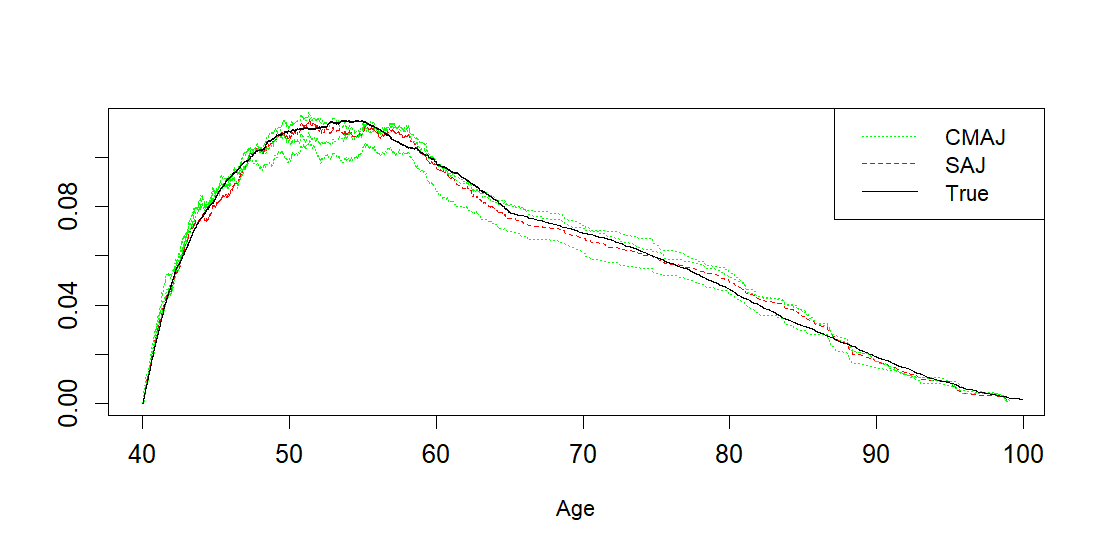"} 	
	\caption{Estimates of the scaled occupation probabilities $p_2^\rho$ (black) for $n=2,000$ using the SAJ (red) and the CMAJ (green) estimators; the latter is illustrated with three different realizations of its auxiliary uniformly distributed random variables.}
	\label{fig:plot_probabilities}
\end{figure}

One immediate observations is that these uniformly distributed random variables induce additional noise to the estimates; apart from that, the estimators exhibit very similar characteristics. This is further illustrated in Figure~\ref{fig:plot_histo}, where we consider $100$ realizations of the auxiliary random variables, but take a snapshot at $t=20$. The immediate and unsurprising conclusion is that the SAJ estimator strictly improves on the CMAJ estimator, since it eliminates the artificial noise introduced by the auxiliary random variables.

\begin{figure}[h]
	\centering
	\includegraphics[width=0.9\textwidth]{"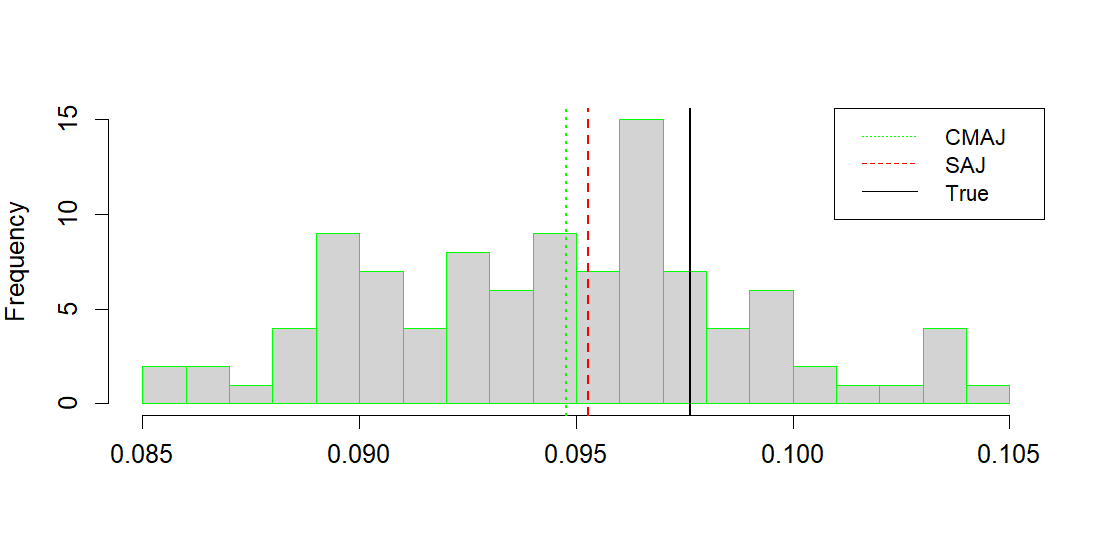"} 	
	\caption{Histogram of estimates of the scaled occupation probability $p_2^\rho(20)$ (black) for $n=2,000$ and $100$ realizations of the auxiliary uniformly distributed random variables together with the resulting average estimate (green) and the SAJ estimate (red).}
	\label{fig:plot_histo}
\end{figure}

We now turn to the comparison of the SAJ estimator with the two-dimensional Aalen--Johansen (2dAJ) estimator. We focus on the expected accumulated cash flow $t \mapsto A(t)$ to avoid comparing apples with pears. From~\eqref{eq:1d_rep_sim_study}, we obtain the plug-in SAJ estimator
\begin{align*}
t \mapsto
b_0
+
\int_0^t \Big(&\mathbb{p}_1(s-) B_1(\mathrm{d}s) + \mathbb{p}_1(s-) b_{13}(s) \mathbb{\Lambda}_{13}(\mathrm{d}s) \\
&+\mathbb{p}_2^\rho(s-) B_2(\mathrm{d}s) + \mathbb{p}_2^\rho(s-)b_{13}(s) \mathbb{\Lambda}_{25}^\rho(\mathrm{d}s)\Big),
\end{align*}
while~\eqref{eq:2d_rep_sim_study} yields the plug-in 2dAJ estimator
\begin{align*}
t &\mapsto
b_0
+
\int_0^t \mathbb{p}_1(s-) \Big(B_1(\mathrm{d}s) + b_{13}(s) \mathbb{\Lambda}_{13}(\mathrm{d}s)\Big) \\
&\quad+
\int_0^t \int_{\boldsymbol{0}}^{(t,s-)} \rho(u_1) \Big(\mathbb{p}_{(1,1)}(\boldsymbol{u}-)\mathbb{\Lambda}_{(1,1)(2,2)}(\mathrm{d}\boldsymbol{u}) \\
&\qquad\qquad\qquad\qquad\qquad - \mathbb{p}_{(1,2)}(\boldsymbol{u}-)\big(\mathbb{\Lambda}_{(1,2)(2,5)}(\mathrm{d}\boldsymbol{u}) + \mathbb{\Lambda}_{(1,2)(2,6)}(\mathrm{d}\boldsymbol{u})\big)\Big) B_2(\mathrm{d}s) \\
&\quad+
\int_{\boldsymbol{0}}^{(t,t)} \rho(u_1) \mathbb{p}_{(1,2)}(\boldsymbol{u}-)b_{25}(u_2) \mathbb{\Lambda}_{(1,2)(2,5)}(\mathrm{d}\boldsymbol{u}).
\end{align*}
The computation of these estimators requires numerical integration, which reduces to a one- and two-dimensional recursion, respectively. In particular, the 2dAJ estimator is computationally more demanding, the complexity being $\mathcal{O}(n^2)$ rather than $\mathcal{O}(n)$. Comparing the two estimators, one notices another fundamental difference: the latter separates $\rho$ from the estimation procedure. If the contractual payments differ between insured, then so do potentially the scaling factors $\rho$, which would necessitate up to $n$ SAJ estimates, increasing the complexity of the SAJ estimator to $\mathcal{O}(n^2)$.

In Figures~\ref{fig:plot_cf_2000} and~\ref{fig:plot_cf_5000}, we have plotted the resulting estimates for $n=2,000$ and $n=5,000$, respectively. In the body, both estimators perform well and essentially identically. However, in the tail, the SAJ estimator outperforms the 2dAJ estimator. This deficiency of the 2dAJ estimator is in line with earlier results for the bivariate survival case reported in, for instance, the last paragraph of Subsection~4.3 in~\cite{GillJohansen1990}. In light hereof, we now give some further theoretical comments following and extending Section~7 in~\cite{Bathke2024}.

\begin{figure}[h]
	\centering
	\includegraphics[width=0.9\textwidth]{"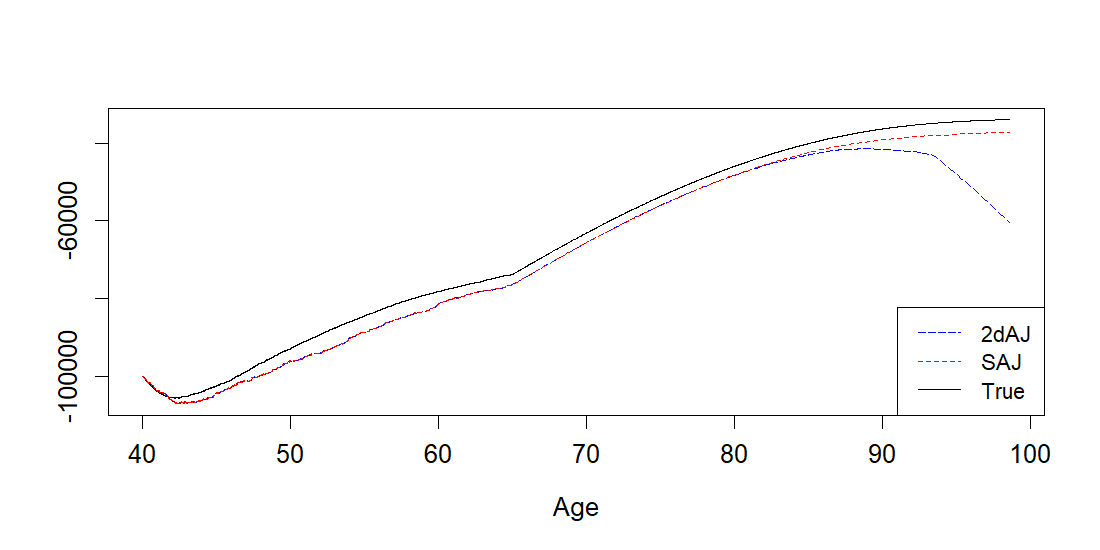"} 	
	\caption{Estimates of the expected accumulates cash flow $t \mapsto A(t)$ (black) for $n=2,000$ using the SAJ (red) and the 2dAJ (blue) estimators.}
	\label{fig:plot_cf_2000}
\end{figure}

\begin{figure}[h]
	\centering
	\includegraphics[width=0.9\textwidth]{"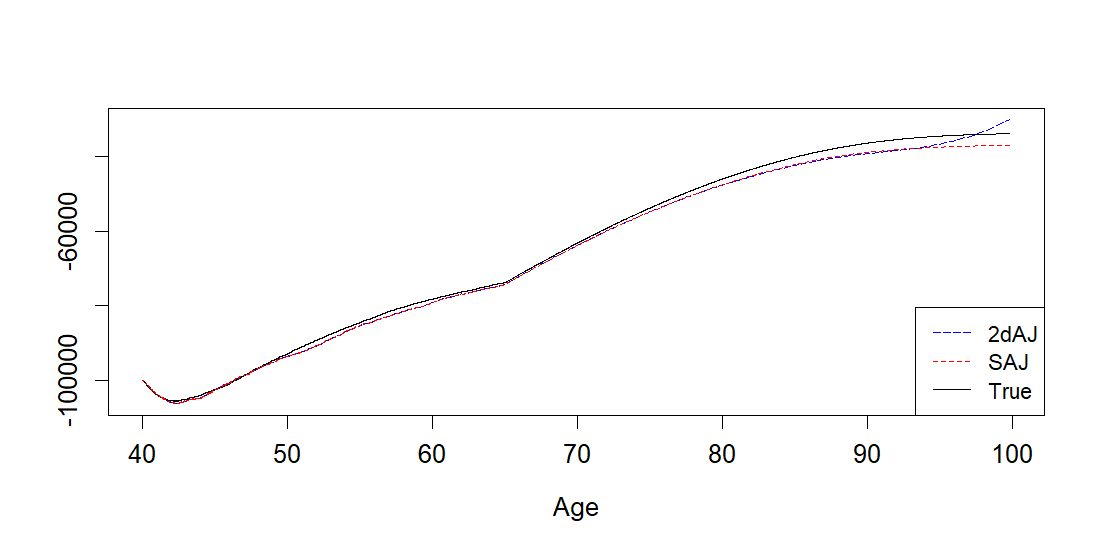"} 	
	\caption{Estimates of the expected accumulates cash flow $t \mapsto A(t)$ (black) for $n=5,000$ using the SAJ (red) and the 2dAJ (blue) estimators.}
	\label{fig:plot_cf_5000}
\end{figure}
Now, in the absence of censoring, the estimators agree and equal the non-parametric maximum likelihood estimator. This follows from the fact that $\mathbb{I}^\rho = \mathbb{p}^\rho$ and $\mathbb{I}_{\boldsymbol{j}} = \mathbb{p}_{\boldsymbol{j}}$ when no censoring is present. Under right-censoring, the 2dAJ estimator may fail to satisfy monotonicity and positivity; this was already noted in~\cite{Dabrowska1988} for the bivariate survival case. It also introduces additional sub-sampling, which is best illustrated via a simple example. Consider three individuals with the following event times and types:
\begin{align*} 
&\text{Individual I:} &&\text{times} = (\tau^1,\sigma^1), &&\text{types} = (1 \to 2, 2 \to 4), \\
&\text{Individual II:} &&\text{times} = (\tau^2,R^2), &&\text{types} = (1 \to 2, \text{right-censoring}), \\
&\text{Individual III:} &&\text{times} = (\tau^3,\sigma^3), &&\text{types} = (1 \to 2, 2 \to 4)
\end{align*}
with $\tau^3 < \tau^1 < \tau^2 < \sigma^1 < R^2 < \sigma^3$. The plug-in SAJ estimate of $A(\sigma^3) - A(R^2)$ becomes $\mathbb{p}_2^{\rho}(R^2)(B_2(\sigma^3)-B_2(R^2))$, where $\mathbb{p}_2^{\rho}(R^2)$ is given by
\begin{align*}
\mathbb{p}_2^{\rho}(R^2)
=
\rho(\tau^3)\frac{\mathbb{p}_1^{\rho}(0)}{\amsmathbb{I}_1^{\rho}(0)}\frac{1}{3}+\rho(\tau^1)\frac{\mathbb{p}_1^{\rho}(\tau^3)}{\amsmathbb{I}_1^{\rho}(\tau^3)}\frac{1}{3}+\rho(\tau^2)\frac{\mathbb{p}_1^{\rho}(\tau^1)}{\amsmathbb{I}_1^{\rho}(\tau^1)}\frac{1}{3}
	-\rho(\tau^1)\frac{\mathbb{p}_2^{\rho}(\sigma^1)}{\amsmathbb{I}_2^{\rho}(\sigma^1)}\frac{1}{3},
\end{align*}
while the plug-in 2dAJ estimate reads 
\begin{align*}
\frac{1}{3}
\bigg(&\rho(\tau^3)\frac{\mathbb{p}_{(1,1)}(0,0)}{\amsmathbb{I}_{(1,1)}(0,0)}+\rho(\tau^1)\frac{\mathbb{p}_{(1,1)}(\tau^3,\tau^3)}{\amsmathbb{I}_{(1,1)}(\tau^3,\tau^3)}\\
&+\rho(\tau^2)\frac{\mathbb{p}_{11}(\tau^1,\tau^1)}{\amsmathbb{I}_{(1,1)}(\tau^1,\tau^1)}-\rho(\tau^1)\frac{\mathbb{p}_{(1,2)}(\tau^1,\sigma^1)}{\amsmathbb{I}_{(1,2)}(\tau^1,\sigma^1)}\bigg)\Big(B_2(\sigma^3) - B_2(R^2)\Big).
\end{align*}
For these three individuals, right-censoring only takes place from state $2$, so the first three summands across estimates agree. The only difference lies in the last summands, that is between
\begin{align*}
\rho(\tau^1)\frac{\mathbb{p}_2^{\rho}(\sigma^1)}{\amsmathbb{I}_2^{\rho}(\sigma^1)} \quad \text{and} \quad \rho(\tau^1)\frac{\mathbb{p}_{(1,2)}(\tau^1,\sigma^1)}{\amsmathbb{I}_{(1,2)}(\tau^1,\sigma^1)},
\end{align*}
where the 2dAJ estimator performs additional sub-sampling.

\subsection{Conclusion}

The SAJ estimator strictly outperforms the CMAJ estimator as it foregoes the latter's auxiliary uniformly distributed random variables that add unnecessary noise. The SAJ and the 2dAJ estimators show comparable performance, but have different strengths and weaknesses. The SAJ estimator has much lower variability in the tail, is easier to implement due to its similarity with known estimators for Markov models, and is computationally far less demanding if the scaling factors do not differ between insured. The 2dAJ estimator is computationally demanding and more difficult to implement, however once implemented it achieves true separation between the contractual payments, including the scaling factor, and the statistical estimation of biometric and behavioral risks. This may be particularly important if the estimates are to find application across insurers, as is commonly the case for actuarial tables.

\section{Extension}\label{sec:extension}

The preceding sections focused on scaled payments with $H$ given by
\begin{align}\label{eq:H_rho}
H(t) = \rho(\tau,Z_{\tau{-}},Z_\tau)^{\mathds{1}_{\{\tau \leq t\}}},
\end{align}
where $\tau$ was interpreted as the exercise time of some policyholder option. This is not the only actuarial example of some importance. If we instead take
\begin{align}\label{eq:H_r}
H(t) = \exp\!\bigg\{-\int_0^t \delta(u) \, \mathrm{d}u\bigg\},
\end{align}
then $B$ given by
\begin{align*}
B(\mathrm{d}t) &= H(t) B^\circ(\mathrm{d}t), &&B(0) = B^\circ(0),
\end{align*}
constitutes the discounted version of $B^\circ$ using the (stochastic) interest rate $\delta$, and $t \mapsto A(t) = \amsmathbb{E}[B(t)]$ becomes the expected accumulated discounted cash flow.

In the following, we explain how the key results of the preceding sections still apply, even if the stochastic process $H$ is only required to be right-continuous, of finite variation, uniformly bounded on compacts, and -- in regards to non-parametric estimation -- adapted to the filtration $\amsmathbb{F}^Z$. In particular, no reference to~\eqref{eq:H_rho} is made. This serves to showcase a broader applicability of our contribution and might, {not least} with~\eqref{eq:H_r} in mind, be of independent interest. {Most contractual payments can be brought into the form $B(\mathrm{d}t) = H(t)B^\circ(\mathrm{d}t)$, so long as only the size and not the `timing' of the sojourn and transition payments is path-dependent. The downside, of course, is that the resulting estimates likely will be contract-specific from the get-go.}

Throughout, we assume that $H$ is adapted to $\amsmathbb{F}$, right-continuous, of finite variation, and uniformly bounded on compacts. We further define
\begin{align}\nonumber
\bar{p}_j(t) &= \amsmathbb{E}[H(t) \mathds{1}_{\{Z_t = j\}}], \\ \label{eq:forward_rate_def1}
\bar{H}_j(\mathrm{d}t) &= \frac{1}{\bar{p}_j(t-)} \amsmathbb{E}[\mathds{1}_{\{Z_{t-} = j\}}H(\mathrm{d}t)], \\ \label{eq:forward_rate_def2}
\bar{\Lambda}_{jk}(\mathrm{d}t) &= \frac{1}{\bar{p}_j(t-)} \amsmathbb{E}[H(t) N_{jk}(\mathrm{d}t)], &&\bar{\Lambda}_{jj} = \bar{H}_j - \sum_{\ell \in \mathcal{J}\atop \ell \neq j} \bar{\Lambda}_{j\ell},
\end{align}
for $j,k \in \mathcal{J}$, $j \neq k$. The following result mirrors Proposition~\ref{prop:rep1d}.
\begin{proposition}\label{prop:rep1d_extended}
It holds that
\begin{align*}
A(t) = \sum_{j\in\mathcal{J}} \bar{p}_j(t-) \bigg({\big(1 + \Delta \bar{H}_j(t)\big)}B_j(\mathrm{d}t) + \sum_{k\in\mathcal{J}\atop k \neq j} b_{jk}(t) \bar{\Lambda}_{jk}(\mathrm{d}t)\bigg)
\end{align*}
and that
\begin{align*}
\frac{
\bar{p}_j(t)
}
{
\amsmathbb{E}[H(0)]
}
&=
\bigg[\Prodi_0^t \big(\emph{Id} + \bar{\Lambda}(\mathrm{d}s)\big)\bigg]_{z_0j}, &&j\in\mathcal{J}.
\end{align*}
\end{proposition}
\begin{proof}
Note that
\begin{align*}
\amsmathbb{E}[H(t) \mathds{1}_{\{Z_{t-}=j\}}]
=
\bar{p}_j(t-)
+
\amsmathbb{E}\Big[\big(H(t)-H(t-)\big)\mathds{1}_{\{Z_{t-}=j\}}]
=
\bar{p}_j(t-)\big(1 + \Delta \bar{H}_j(t)\big).
\end{align*}
It then immediately follows that
\begin{align*}
A(t) = \sum_{j\in\mathcal{J}} \bar{p}_j(t-) \bigg({\big(1 + \Delta \bar{H}_j(t)\big)}B_j(\mathrm{d}t) + \sum_{k\in\mathcal{J}\atop k \neq j} b_{jk}(t) \bar{\Lambda}_{jk}(\mathrm{d}t)\bigg),
\end{align*}
which verifies the first statement of the proposition. The proof of the second statement mirrors the proof of Proposition~\ref{prop:rep1d}, but requires the following key identity:
\begin{align}\label{eq:key_identity_generalized}
\begin{split}
H(t) \mathds{1}_{\{Z_t = j\}}
&=
\mathds{1}_{\{j=z_0\}}H(0)
+
\sum_{k \in \mathcal{J}\atop k \neq j} \int_0^t H(s) N_{kj}(\mathrm{d}s) \\
&\quad-
\sum_{k \in \mathcal{J}\atop k \neq j} \int_0^t H(s) N_{jk}(\mathrm{d}s)
+
\int_0^t \mathds{1}_{\{Z_{s-}=j\}} H(\mathrm{d}s),
\end{split}
\end{align}
which follows from~\eqref{eq:key_identity} and an application of integration by parts.
\end{proof}
\begin{remark}\label{rmk:forward_eq_extended}
Another way to phrase the second part of the theorem is that $(\bar{p}_j)_j$ uniquely solves the forward integral equations
\begin{align*}
\bar{p}_j(\mathrm{d}t) &= \sum_{k \in \mathcal{J}} \bar{p}_k(t-) \bar{\Lambda}_{kj}(\mathrm{d}t), &&j \in \mathcal{J},
\end{align*}
with initial conditions $\bar{p}_j(0) = \mathds{1}_{\{j = z_0\}}\amsmathbb{E}[H(0)]$.
\end{remark}
\begin{remark}
The product integral representation of $(\bar{p}_j)_j$, and the equivalent forward integral equations, constitute an extension of previous insights for Markov processes to the non-Markov regime; confer, for instance, with Appendix~1 in~\cite{BuchardtFurrerMoller2020} as well as Theorem~3.1 in~\cite{AhmadBladt2023}.
\end{remark}
\begin{remark}
If $H$ is taken according to~\eqref{eq:H_r}, then the equations~\eqref{eq:forward_rate_def1}--\eqref{eq:forward_rate_def2} yield the correct Ansatz for a proper constructive and simultaneous definition of forward interest and forward transition rates in the spirit of Definition~4.1 in~\cite{Buchardt2017}, with the forward integral equations given in Remark~\ref{rmk:forward_eq_extended} offering the appropriate extension of Theorem~4.4 in~\cite{Buchardt2017}. This addresses an open question that dates back to at least~\cite{MiltersenPersson2005}; confer also, for instance, with~\cite{Norberg2010,Buchardt2014,ChristiansenNiemeyer2015,BuchardtFurrerSteffensen2019}.
\end{remark}
We now turn to non-parametric estimation under the assumption that $H$ is adapted to $\amsmathbb{F}^Z$. We shall not carry out all arguments in complete detail, but just report on the key findings. If $H$ is taken according to for instance~\eqref{eq:H_r}, one could readily ask why assuming $H$ to be adapted to $\amsmathbb{F}^Z$ has merit compared to, say, just requiring $H$ (and therefore $\delta$) to be deterministic. However, state-dependent interest rates may appear when dealing with certain reserve-dependent expenses, see Section~G in~\cite{Norberg1991} for the basic idea as well as~\cite{ChristiansenDenuitDhaene2014} and Section~4 in~\cite{BuchardtFurrerMoller2020} for further details. Therefore, even under the specification~\eqref{eq:H_r}, the case of $\amsmathbb{F}^Z$-adaptedness carries practical value.

As in Section~\ref{sec:num}, we let the right-censoring time $R$ be independent of $Z$ and suppose $((Z_t)_{0\leq t\leq R}, \eta \wedge R)$ is observed. In particular, the process $H$ is sufficiently observed. Due to right-censoring being entirely random, we obtain the key identities
\begin{align*}
\bar{H}_j(\mathrm{d}t) &= \frac{1}{\bar{p}_j^\texttt{c}(t-)} \amsmathbb{E}[\mathds{1}_{\{Z_{t-} = j\}}H(\mathrm{d}t \wedge R)], \\
\bar{\Lambda}_{jk}(\mathrm{d}t) &= \frac{1}{\bar{p}_j^\texttt{c}(t-)} \amsmathbb{E}[H(t) N_{jk}(\mathrm{d}t \wedge R)],
\end{align*}
for $t$ below $\theta$ and where we define $\bar{p}_j^\texttt{c}(t) = \amsmathbb{E}[H(t)\mathds{1}_{\{Z_t = j\}}\mathds{1}_{\{t < R\}}]$, $j\in\mathcal{J}$; the identities may be compared to~\eqref{eq:lambda1d_c}.

Based hereon, we may for an i.i.d.\ sample directly define estimators of $\bar{\Lambda}$ and $\bar{p}$ based on averages of the above quantities, and just like their counterparts in Section~\ref{sec:num}, these estimators will be strongly consistent and asymptotically normal. This is due to the fact that the statements and proofs of Section~\ref{sec:num} only to a very limited degree rely on the specification~\eqref{eq:H_rho} and when they do, such as in~\eqref{eq:key_1dJ0}, the statements and proofs are easily adapted to the general case based on~\eqref{eq:key_identity_generalized}.

\section*{Acknowledgments}

Significant parts of the research presented in this paper were carried out while Christian Furrer was a Junior Fellow at the Hanse-\!Wissenschaftskolleg Institute for Advanced Study in Delmenhorst, Germany. The authors are grateful to Marcus C.\ Christiansen for initiating and facilitating their fruitful research collaboration. Christian Furrer also thanks Martin Bladt for valuable discussions on empirical process theory. 

\appendix 
\bibliographystyle{amsplain}
\bibliography{references.bib}

\end{document}